%% file: unav93.tex
\documentclass[11pt]{article}

\usepackage{amsmath, amssymb, latexsym, amsthm}
\usepackage{mathabx}
\usepackage{bm}
\usepackage{wasysym}
\usepackage{fullpage}
\usepackage{color}
\usepackage{tikz}
\usepackage{graphicx}
\usepackage{stmaryrd,amsbsy,enumerate}
\usepackage{hyperref}
\usepackage{mathscinet}
\usetikzlibrary{calc}
\usepackage{cancel}
\usepackage{xspace}
\usepackage{url}
\usepackage{longtable}
\usepackage[obeyFinal,colorinlistoftodos,textsize=footnotesize]{todonotes}
\usepackage{extarrows}
\usepackage{enumitem}
\usepackage{mathtools,amsmath}
\usepackage{array}

\usepackage[margin=10pt,font=small,labelfont=sc,labelsep=period]{caption} [2023/03/12]

\usepackage{soul}

\newcommand{\ai}{\text{\hglue 0.06 cm}
 {\mathchoice
  { \includegraphics[height = 1.2ex]{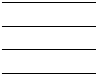}}
  { \includegraphics[height = 1.2ex]{li1}}
  { \includegraphics[height = 1.2ex]{li1}}
  { \includegraphics[height = 1.2ex]{li1}}
 }\text{\hglue 0.06 cm}}

\newcommand*{\Scale}[2][4]{\scalebox{#1}{\ensuremath{#2}}}%

\colorlet{mygray}{gray!60}

\def\BO#1{{\ov{\bom{#1}}}}

\def\oA{{\overline{A}}}
\def\oB{{\overline{B}}}
\def\oC{{\overline{C}}}
\def\Rev#1{{\,{-}#1}}
\def\sphere{{\mathbb S^{2}}}
\def\bom#1{{\boldsymbol{#1}}}

\def\b1{{\bom{1}}}
\def\b2{{\bom{2}}}

\def\perm#1{{\llbracket{#1}\rrbracket}}
\def\cycper#1{{\llparenthesis{\,}{#1}{\,}\rrparenthesis}}

\def\rotvU#1#2#3{{\text{\rm rot}_{#1}(#2,#3)}}
\def\rotvU#1#2#3{{\circlearrowright{\text{\hglue -0.05cm}}({#1},{#2};{#3})}}

\def\edges#1#2{{ {\boldsymbol{#1}}{\boldsymbol{:}}{\boldsymbol{#2}} }}
\def\edges#1#2{{#1}{\ai}{#2}}
\def\Edges#1{{\ai}{#1}{\ai}}
\def\Edges#1{{}{#1}{\ai}}

\def\ov#1{{\overline{#1}}}

\def\blueit#1{\textcolor{black}{#1}}

\def\hh{{\mathcal H}}

\def\b#1{{\mathbf{#1}}}

\def\outer{convex\xspace}
\def\Outer{Convex\xspace}
\def\outerity{convexity\xspace}

\def\outer{\blueit{1-page}\xspace}
\def\Outer{\blueit{1-page}\xspace}
\def\outerity{\blueit{1-pageness}\xspace}

\def\outerd{{\blueit{drawing}}}

\newtheorem{theorem}{Theorem} 
\newtheorem{theorem*}{Theorem} 
\newtheorem{proposition}[theorem]{Proposition}

\newtheorem{corollary}[theorem]{Corollary}
\newtheorem{lemma}[theorem]{Lemma}

\newtheorem{remark}[theorem]{Remark}

\theoremstyle{definition}

\newtheorem{observation}[theorem]{Observation}

\newtheorem{definition}{Definition}

\begin{document}


\title{The unavoidable drawings of complete multipartite graphs}
\author{J\'ozsef Balogh\thanks{University of Illinois Urbana-Champaign, 1409 W. Green Street, Urbana IL 61801, United States, and Extremal Combinatorics and Probability Group (ECOPRO), Institute for Basic Science (IBS), Daejeon, South Korea. {\em E-mail:} {\tt jobal@illinois.edu}}  \and Irene Parada\thanks{Polytechnic University of Catalonia. Barcelona, Spain. {\em E-mail:} {\tt irene.parada@upc.edu}} \and Gelasio Salazar\thanks{Instituto de F\'\i sica, Universidad Aut\'onoma de San Luis Potos\'{\i}. San Luis Potos\'{\i}, Mexico. {\em E-mail:} {\tt gelasio.salazar@uaslp.mx}}}


\maketitle

\begin{abstract}
In a simple drawing of a graph every pair of edges intersect each other in at most one point, which is either a common endvertex or a proper crossing. For each positive integer $n$, Negami identified a drawing $B_n$ of the complete bipartite graph $K_{n,n}$, and proved that if $N$ is sufficiently large, then every drawing of $K_{N,N}$ contains a drawing of $K_{n,n}$ weakly isomorphic to $B_n$. Thus $B_n$ is (up to weak isomorphism) the only {\em unavoidable} drawing of $K_{n,n}$. We extend this result to complete multipartite graphs, characterizing their unavoidable drawings. 
\end{abstract}

\section{Introduction}\label{sec:introduction}

In a {\em simple} drawing of a graph on the plane or on the sphere (i) vertices are represented by distinct points; (ii) each edge is represented by a Jordan curve whose endpoints are its corresponding endvertices; (iii) no vertex lies in the interior of an edge; (iv) every pair of edges intersect each other at most once, either at a common endvertex or at a crossing; and (v) no three edges have a common crossing. The study of simple drawings, also known as \emph{good drawings} and as \emph{simple topological graphs}, 
and their subdrawings, has attracted and continues to attract significant interest~\cite{shootingstars,twisted,fulekvargas,
kynclimproved,pst,tangled-thrackle,sukzeng}.

All drawings under consideration are implicitly assumed to be simple. Also, all drawings considered are hosted on the sphere $\sphere$. Working on the sphere instead of on the plane simplifies the arguments at certain points, and our main results clearly also hold for drawings on the plane.

Let $D$ be a drawing of a graph $G$. If $H$ is a subgraph of $G$, then $D$ naturally induces a drawing of $H$, simply by removing from $D$ all the vertices and edges not in $H$. This is a {\em subdrawing} of $D$. Subdrawings {\em induced} by sets of edges or vertices occur often throughout this work. If $E$ is a set of edges of $G$, then $D[E]$ denotes the subdrawing of $D$ that contains the edges in $E$ and their endvertices. \blueit{Similarly, if $U$ is a set of vertices of $G$, then $D[U]$ denotes the subdrawing of $D$ that contains $U$ and the edges that have both endvertices in $U$.}



We recall that two drawings $D_1,D_2$ of a graph are {\em weakly isomorphic} if there is an incidence-preserving bijection between the drawings such that two edges cross each other in $D_1$ if and only if their images in $D_2$ cross each other. This notion is indeed weaker than isomorphism: $D_1$ and $D_2$ are {\em isomorphic} if there is a self-homeomorphism of $\sphere$ that takes $D_1$ to $D_2$.

\subsection{Background: the unavoidable drawings of $K_n$ and $K_{n,n}$} Pach, Solymosi and T\'oth~\cite{pst} identified for each integer $n>0$ two drawings $C_n$ and $T_n$ of the complete graph $K_n$ (see Figure~\ref{fig:harborth}), and proved that $C_n$ and $T_n$ are the only ``unavoidable'' drawings of $K_n$, in the following sense. 

\begin{theorem}[The unavoidable drawings of $K_n$~\cite{pst}]\label{thm:pst}
Let $n >0$ be an integer. If $N$ is sufficiently large, then every drawing of $K_N$ contains a drawing of $K_n$ weakly isomorphic to $C_n$ or $T_n$.
\end{theorem}

\begin{figure}[ht!]
\def\ta#1{{\Scale[1.6]{#1}}}
\centering
\scalebox{0.44}{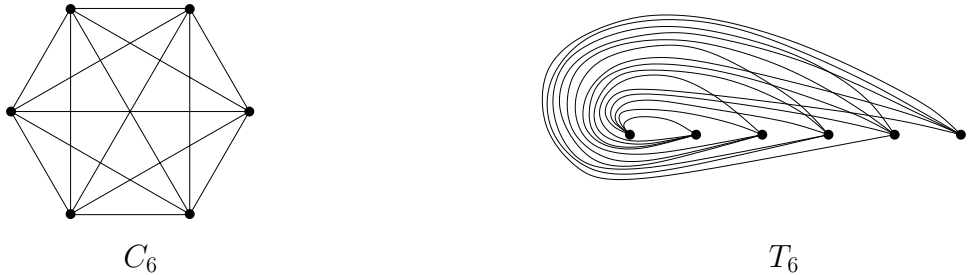}
\caption{The drawings $C_6$ and $T_6$ of the complete graph $K_6$. It is straightforward to generalize these constructions to obtain drawings $C_n$ and $T_n$ for any positive integer $n$. These are the {unavoidable} drawings of $K_n$.}
\label{fig:harborth}
\end{figure}

For a recent refinement on the bounds in~\cite{pst} we refer the reader to~\cite{sukzeng}.
For straight-line drawings of $K_n$, it is a consequence of the 
celebrated Erd\H{o}s-Szekeres theorem~\cite{erdos1935combinatorial}
that $C_n$ is the only unavoidable drawing.

In a similar vein, Negami proved in~\cite{negami} that there is a unique (up to weak isomorphism) unavoidable drawing $B_n$ for complete bipartite graphs. See Figure~\ref{fig:165} for an illustration of $B_4$. 

\begin{theorem}[The unavoidable drawings of $K_{n,n}$~\cite{negami}]\label{thm:negami}
Let $n>0$ be an integer. If $N$ is sufficiently large, then every drawing of $K_{N,N}$ contains a drawing of $K_{n,n}$  weakly isomorphic to $B_n$.
\end{theorem}

\begin{figure}[ht!]
\def\ta#1{{\Scale[2.0]{#1}}}
\def\tb#1{{\Scale[0.8]{#1}}}
\centering
\scalebox{0.48}{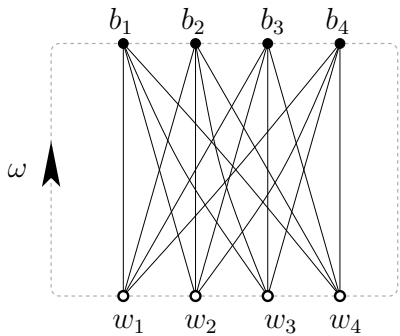}
\caption{The drawing $B_4$ of $K_{4,4}$. It is straightforward to generalize this construction to obtain a drawing $B_n$ of $K_{n,n}$, for any positive integer $n$. These are the unavoidable drawings of $K_{n,n}$. The bounding curve $\omega$ witnesses that this is a \outer drawing: all vertices lie on $\omega$, and all edges lie on the same connected component of $\sphere\setminus\omega$. As we traverse $\omega$ in the direction shown (so that the edges stay to our right during this traversal) we encounter the vertices in the cyclic order $b_1,b_2,b_3,b_4,w_4,w_3,w_2,w_1$. Thus, this is a $\cycper{b_1,b_2,b_3,b_4,w_4,w_3,w_2,w_1}$-\outerd. If for brevity we let $B$ be the permutation $\perm{b_1,b_2,b_3,b_4}$ and let $W$ be $\perm{w_1,w_2,w_3,w_4}$, then we may equivalently say that this is a $\cycper{B\cdot W^{-1}}$-\outerd.}
\label{fig:165}
\end{figure}


\section{Our main result: the unavoidable drawings \\ of complete multipartite graphs}\label{sec:unavoidable}

Our main result is the generalization of Theorem~\ref{thm:negami} to complete multipartite graphs. Throughout this paper $K_n^m$ denotes the complete multipartite graph with $m$ classes, each of size $n$. For each pair $m,n$ of positive integers we find a collection of drawings of $K_n^m$ (we call them {\em canonical drawings}) with the following property: if $N$ is sufficiently large compared to $n$, then every drawing of $K_N^m$ contains a canonical drawing of $K_n^m$. Thus, in the spirit of Theorems~\ref{thm:pst} and~\ref{thm:negami}, the canonical drawings are precisely the unavoidable drawings of $K_n^m$. This is formally stated in Theorem~\ref{thm:main} at the end of this section.

In contrast to the canonical drawings of $K_n$ and $K_{n,n}$, the canonical drawings of $K_n^m$ form a much richer family.
Still, they are quite natural. 
To get a flavour of canonical drawings we entice the reader to take a sneak peek at Figure~\ref{fig:180}, where we illustrate a canonical drawing of $K_3^5$.

Before we can proceed to the formal definition of a canonical drawing we need to explain a crucial convention that we follow throughout this paper: partite classes of $K_n^m$ are regarded not only as sets of vertices, but as permutations (that is, ordered sets) of vertices. This is explained in Section~\ref{sub:permutations}. As we shall see, canonical drawings are described in terms of \outer drawings, and we devote Section~\ref{sub:outerplanar} to review this notion. 


Before we formally define what a canonical drawing is, in Section~\ref{sub:example} we thoroughly analyze the canonical drawing in Figure~\ref{fig:180}. Then we formally introduce canonical drawings in Section~\ref{sub:canonical}, and in Section~\ref{sub:main} we state the main result in this paper. 

We emphasize that even though after Section~\ref{sub:outerplanar} it is possible to skip Section~\ref{sub:example} and proceed right away to Section~\ref{sub:canonical}, we strongly encourage the reader to go through Section~\ref{sub:example}, where we discuss at leisure the main features of a canonical drawing.

\subsection{Partite classes of $K_n^m$ as permutations of vertices}\label{sub:permutations} 

In all complete multipartite graphs we consider {\em we assume that each partite class is not only a set of vertices, but a permutation (that is, an ordered set) of vertices.} Thus, we regard partite classes both as ordinary (unordered) sets and as permutations (ordered sets), depending on the discussion. 

Throughout this work we use $\perm{\, }$ to denote a permutation, and $\cycper{ \,}$ to denote a cyclic permutation.

Suppose for instance that we have a partite class labelled $\bom{1}$, with vertices $1(1),\ldots,$ $1(n)$. Regarding $\bom{1}$ as a permutation $\bom{1}=\perm{1(1),\ldots,1(n)}$ allows us to consider for instance the reverse permutation $\Rev{\bom{1}}=\perm{1(n),\ldots,1(1)}$. If $\bom{2}=\perm{2(1),\ldots,2(n)}$ is another partite class, then we can consider the concatenation of $\bom{1}$ and $\bom{2}$, for which we use the symbol $\cdot$, namely $\bom{1}\cdot {\bom{2}}=  \llbracket  1(1),\ldots,1(n),2(1),$ $\ldots,2(n)\rrbracket$, or the concatenation of $\bom{1}$ and $\Rev{\bom{2}}$, namely $\bom{1}\cdot \Rev{\bom{2}}=\llbracket{1(1),\ldots,1(n),}$  $2(n),\ldots,2(1)\rrbracket$. 

To see the convenience of this notation, we refer the reader to Figure~\ref{fig:225}, where we reproduce the subdrawing of the drawing in Figure~\ref{fig:180} induced by the edges incident with the vertex set $\{1(1),1(2),1(3)\}$. Let $\bom{1}=\perm{1(1),1(2),1(3)}$, let $\bom{2}=\perm{2(1),2(2),2(3)}$, and let $\bom{3}=\perm{3(1),3(2),3(3)}$. As we traverse the blue curve $\phi_1$ so that the edges stay at our right, we encounter the vertices in the cyclic order $\llparenthesis 1(1),1(2),1(3),3(1),$ $3(2),3(3)$, $2(3),2(2),2(1) \rrparenthesis$. With the notation we introduced this cyclic permutation can be written simply as $\cycper{\bom{1} \cdot \bom{3} \cdot \Rev{\bom{2}}}$.

A {\em subpermutation} $A'$ of a permutation $A$ is obtained by removing (maybe zero) elements from $A$. We write $A'\preceq A$ to denote that $A'$ is a subpermutation of $A$. For instance, if $\bom{1}$ is the permutation in the previous paragraph, then $\perm{1(1),1(3)}\preceq \bom{1}$. 

Finally, if $A$ is a permutation and $a$ precedes $a'$ in $A$, then we write $a <_{A} a'$. 

\subsection{\Outer drawings}\label{sub:outerplanar}

The description of canonical drawings relies crucially on the notion of a \outer drawing: as we will see, the canonical drawings of $K_n^m$ have {\outer} (sub)drawings as their building blocks.

We recall that in a {\em \outer} drawing of a graph (such as the drawing in Figure~\ref{fig:165}) there is a simple closed curve $\omega$ that goes through all the vertices, and has the property that all the edges are contained in one of the two connected components of $\sphere\setminus\omega$. We say that $\omega$ is a {\em bounding curve} for the drawing. 

If as we traverse $\omega$ so that the edges stay at our right during the traversal we encounter the vertices in the cyclic order $v_1,\ldots,v_n$, then we simply say that this is a $\cycper{v_1,\ldots,v_n}$-{\em\outerd}, and that the drawing has {\em bounding order} $\cycper{v_1,\ldots,v_n}$. 

Thus the drawing in Figure~\ref{fig:165} is a $\cycper{b_1,b_2,b_3,b_4,w_4,w_3,w_2,w_1}$-drawing. If $B=\perm{b_1,b_2,b_3,b_4}$ and $W=\perm{w_1,w_2,w_3,w_4}$, then we may equivalently say that this is a $\cycper{B\cdot W^{-1}}$-\outerd. This compact way to describe a \outer drawing will be heavily used for the rest of this paper. For an additional example we refer the reader to Figure~\ref{fig:445} and its caption.

\subsubsection{Negami's theorem in the notation and terminology of \outer drawings}\label{sec:negamit}
In order to get acquainted with the upcoming arguments it seems worth discussing Theorem~\ref{thm:negami} in the notation and terminology of \outer drawings. To simplify our discussions we use the following notation for the rest of this paper. 

\vglue 0.3 cm
\noindent{\bf Notation. }{\sl If $U$ is a set of vertices, then we use $\Edges{U}$ to denote the set of edges that are incident with a vertex in $U$. If $V$ is a set of vertices disjoint from $U$, then we use $\edges{U}{V}$ to denote the set of edges that have one endvertex in $U$ and one endvertex in $V$.}
\vglue 0.3 cm

A glance at the proof of~\cite[Theorem 5]{negami} (which is Theorem~\ref{thm:negami} above) reveals that Negami showed a result formally stronger than Theorem~\ref{thm:negami}. Indeed, his arguments actually prove the following:

\begin{theorem}[\cite{negami}]\label{thm:negami2}
Let $N\ge n>0$ be integers. Let $D$ be a drawing of $K_{N,N}$ with partite classes ${\BO{1}}$ and ${\BO{2}}$. If $N$ is sufficiently large compared to $n$, then there exist $\bom{1}\preceq \BO{1}$ and $\bom{2}\preceq \BO{2}$, with $|\bom{1}|=|\bom{2}|=n$, such that $D[\edges{\bom{1}}{\bom{2}}]$ is either:

\begin{itemize}

\item a $\cycper{\bom{1} \cdot \bom{2}}$-drawing; or 

\item a $\cycper{\Rev{\bom{1}} \cdot \bom{2}}$-drawing; or 

\item a $\cycper{\bom{1} \cdot \Rev{\bom{2}}}$-drawing; or 

\item a $\cycper{\Rev{\bom{1}} \cdot \Rev{\bom{2}}}$-drawing. 

\end{itemize}

\end{theorem}

It is worth emphasizing why Theorem~\ref{thm:negami2} is indeed (slightly) stronger than Theorem~\ref{thm:negami}. Suppose for definiteness that $n=4$. Suppose that $\BO{1}$ is the permutation $\perm{1(1),\ldots,1(N)}$ and $\BO{2}$ is the permutation $\perm{2(1),\ldots,2(N)}$. 

Theorem~\ref{thm:negami2} then claims (if $N$ is sufficiently large) the existence of {\em subpermutations} (not only subsets) $\bom{1}=\perm{1(a_1),1(a_2),1(a_3),1(a_4)}$ of $\BO{1}$ (thus  $a_1< a_2 < a_3 < a_4$) and $\bom{2}=\perm{2(b_1),2(b_2),2(b_3),2(b_4)}$ of $\BO{2}$ (thus $b_1< b_2 < b_3 < b_4$) such that $D[\edges{\bom{1}}{\bom{2}}]$ is of one of the four types in Theorem~\ref{thm:negami2}. These four possible outcomes are illustrated in Figure~\ref{fig:445}(a), (b), (c), and (d), respectively.

\begin{figure}[ht!]
\def\ta#1{{\Scale[2.0]{#1}}}
\def\taaone{{\Scale[1.8]{1(a_1)}}}
\def\taatwo{{\Scale[1.8]{1(a_2)}}}
\def\taathree{{\Scale[1.8]{1(a_3)}}}
\def\taafour{{\Scale[1.8]{1(a_4)}}}
\def\tabone{{\Scale[1.8]{2(b_1)}}}
\def\tabtwo{{\Scale[1.8]{2(b_2)}}}
\def\tabthree{{\Scale[1.8]{2(b_3)}}}
\def\tabfour{{\Scale[1.8]{2(b_4)}}}
\def\tarm#1{{\Scale[2.0]{\text{\rm #1}}}}
\def\tb#1{{\Scale[0.8]{#1}}}
\centering
\scalebox{0.5}{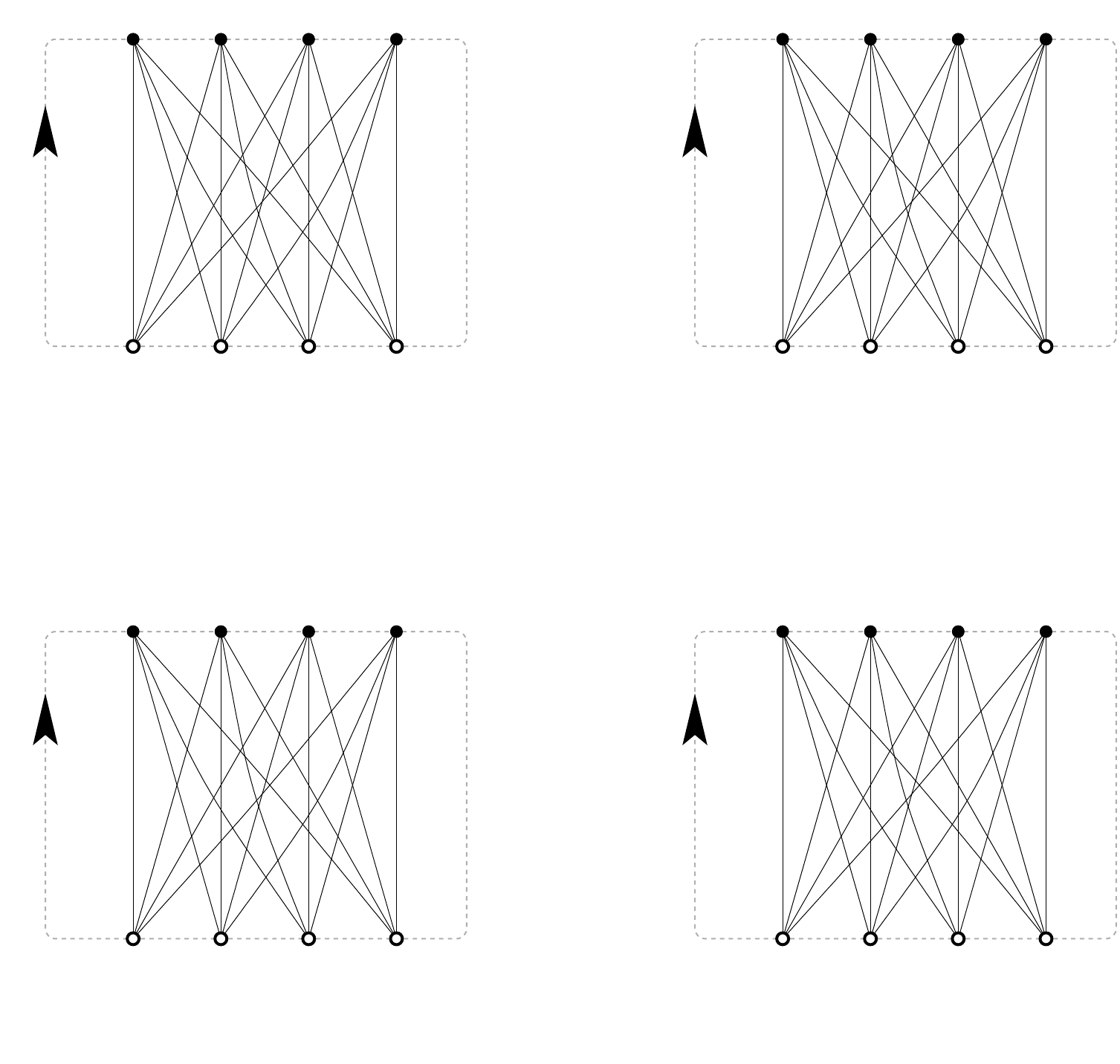}
\caption{Let $\bom{1}=\perm{1(a_1),1(a_2),1(a_3),1(a_4)}$ and $\bom{2}=\perm{2(b_1),2(b_2),2(b_3),2(b_4))}$. Then the drawing in (a) is a $\cycper{\bom{1}\cdot\bom{2}}$-{\outerd}, the drawing in (b) is a $\cycper{\Rev{\bom{1}}\cdot\bom{2}}$-{\outerd}; the drawing in (c) is a $\cycper{\bom{1}\cdot\Rev{\bom{2}}}$-{\outerd}; and the drawing in (d) is a $\cycper{\Rev{\bom{1}}\cdot\Rev{\bom{2}}}$-{\outerd}.}
\label{fig:445}
\end{figure}

As illustrated in Figure~\ref{fig:445}, each of these outcomes in particular implies that \blueit{$D[\edges{\bom{1}}{\bom{2}}]$} is weakly isomorphic to the unavoidable drawing $B_4$ (see Figure~\ref{fig:165}), as stated in Theorem~\ref{thm:negami}. 

\subsubsection{Writing Theorem~\ref{thm:negami2} using a sign function $\sigma$}

We close this section by writing Theorem~\ref{thm:negami2} in a more compact way, in order to introduce the reader to the notion of a {sign function}. As customary, for any positive integer $m$ we use $[m]$ to denote the set $\{1,2,\ldots,m\}$.

As we shall see, if $m > 1$ is an integer then a {\em sign function over $[m]$} is simply a function $\sigma$ that assigns to each ordered pair $(i,j)$ of distinct integers in $[m]$ an integer $\sigma(i,j)$ in $\{-1,1\}$. For $m=2$, a sign function $\sigma$ simply assigns to $\sigma(1,2)$ either $-1$ or $1$, and it assigns to $\sigma(2,1)$ either $-1$ or $1$.

With this notion in hand, Theorem~\ref{thm:negami2} can be equivalently stated as follows.

\begin{theorem}[Theorem~\ref{thm:negami2} using a sign function]\label{thm:negami3}
Let $N\ge n>0$ be integers. Let $D$ be a drawing of $K_{N,N}$ with partite classes ${\BO{1}}$ and ${\BO{2}}$, where $|\BO{1}|=|\BO{2}|=N$. If $N$ is sufficiently large compared to $n$, then there exist $\bom{1}\preceq \BO{1}$ and $\bom{2}\preceq \BO{2}$, with $|\bom{1}|=|\bom{2}|=n$, and a sign function $\sigma$ over $[2]$, such that $D[\edges{\bom{1}}{\bom{2}}]$ is a $\cycper{\sigma(2,1)\,\,\bom{1}\cdot \sigma(1,2)\,\,\bom{2}}$-\outerd.
\end{theorem}









\subsection{An example: a canonical drawing exhaustively analyzed}\label{sub:example}

In this section we exhibit an example of a canonical drawing, and exhaustively analyze its main features. Our purpose is to motivate the properties that appear in the definition of a canonical drawing, which is given in Section~\ref{sub:canonical}. 

We emphasize that this section is entirely optional, as all the notions, notation, and terminology in the definition of a canonical drawing have been already laid out at this point. Nevertheless, we strongly believe that it is worth understanding at depth this particular example before getting to the formal definition of a canonical drawing.

We refer the reader to Figure~\ref{fig:180}, where we depict a canonical drawing $L$ of $K_3^5$. The partite classes are $\bom{1}=\perm{1(1),1(2),1(3)}$, $\bom{2}=\perm{2(1),2(2),2(3)}$, $\bom{3}=\perm{3(1),3(2),3(3)}$, $\bom{4}=\perm{4(1),4(2),4(3)}$, and $\bom{5}=\perm{5(1),5(2),3(3)}$. We emphasize that the gray \blueit{boxes and thick segments} are not part of the drawing, and are included for a later discussion.

\def\inca{{\Scale[3]{\text{\rm (a)}}}}
\def\incb{{\Scale[3]{\text{\rm (b)}}}}
\def\incc{{\Scale[3]{\text{\rm (b)}}}}
\def\incd{{\Scale[3]{\text{\rm (b)}}}}
\def\te#1{{\Scale[4.2]{#1}}}
\def\taf#1{{\Scale[12]{#1}}}
\def\tf#1{{\Scale[8]{#1}}}
\def\tg#1{{\Scale[5.5]{#1}}}
\def\otg#1{{\Scale[6.5]{#1}}}
\def\somea{{\Scale[1.5]{\text{\rm (a)}}}}
\def\someb{{\Scale[1.5]{\text{\rm (b)}}}}
\def\somec{{\Scale[1.5]{\text{\rm (c)}}}}
\def\somed{{\Scale[1.5]{\text{\rm (d)}}}}
\def\somee{{\Scale[1.5]{\text{\rm (e)}}}}
\def\Za{{\Scale[4.0]{E_{u\OL{v}}}}}
\def\Zb{{\Scale[4.0]{E_{\OL{u22}\OL{v}}}}}
\def\Zc{{\Scale[4.0]{E_{\OL{u}{v}}}}}
\def\Rec{{\Scale[6.0]{\rho}}}
\def\tq#1{{\Scale[0.8]{#1}}}
\begin{figure}[ht!]
\centering
\scalebox{0.38}{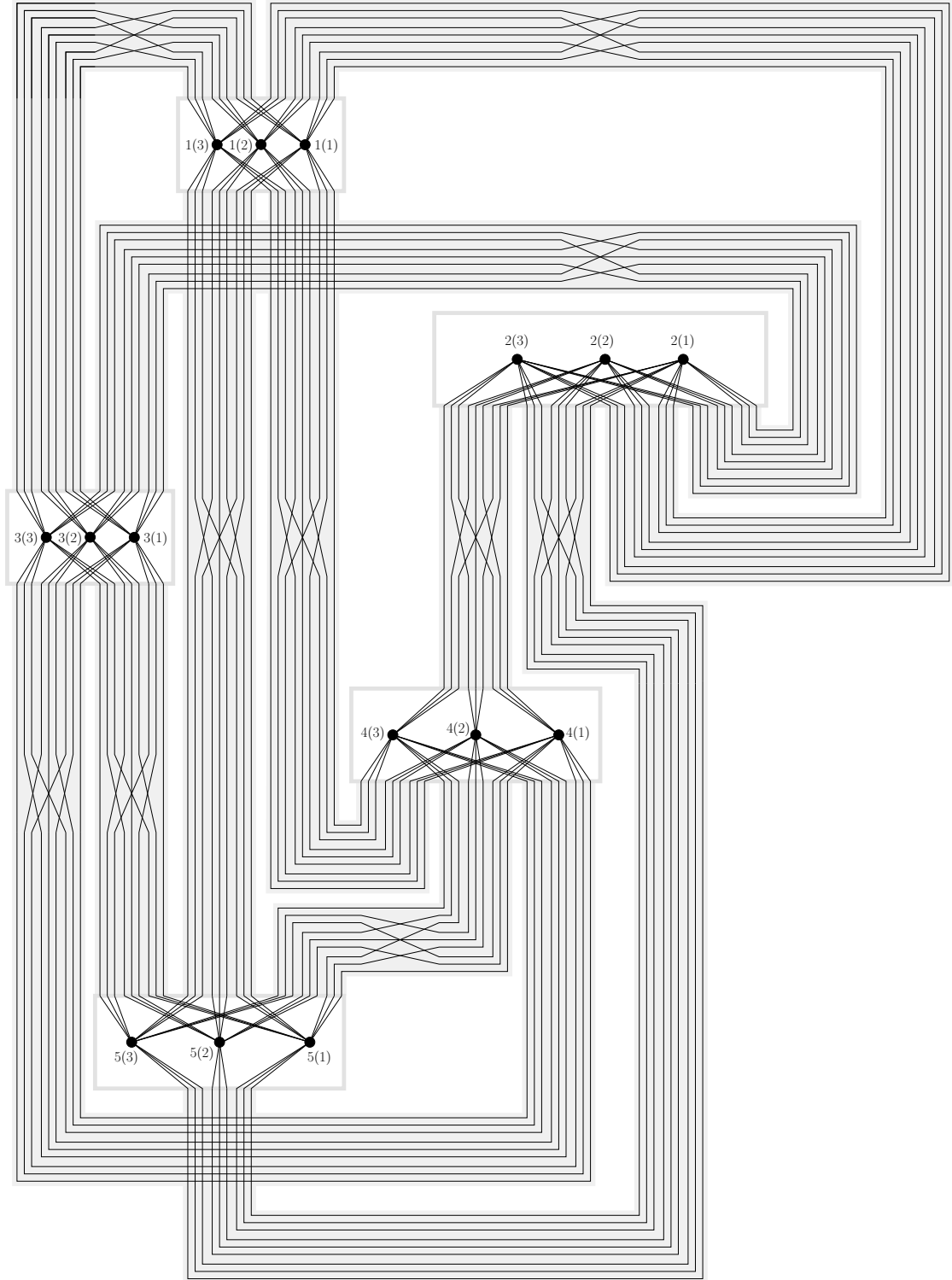}
\caption{{A canonical drawing $L$ of $K_3^5$.}} 
\label{fig:180}
\end{figure}

\blueit{A central feature of the drawing in Figure~\ref{fig:180} is captured in Figure~\ref{fig:225}, where we reproduce the subdrawing induced by the edges incident with $\bom{1}$, that is, the subdrawing $L[\Edges{\bom{1}}]$ induced by $\Edges{\bom{1}}$.}


\def\inca{{\Scale[3]{\text{\rm (a)}}}}
\def\incb{{\Scale[3]{\text{\rm (b)}}}}
\def\te#1{{\Scale[4.2]{#1}}}
\def\taf#1{{\Scale[12]{#1}}}
\def\tf#1{{\Scale[8]{#1}}}
\def\tg#1{{\Scale[5.5]{#1}}}
\def\otg#1{{\Scale[6.5]{#1}}}
\def\somea{{\Scale[1.5]{\text{\rm (a)}}}}
\def\someb{{\Scale[1.5]{\text{\rm (b)}}}}
\def\somec{{\Scale[1.5]{\text{\rm (c)}}}}
\def\somed{{\Scale[1.5]{\text{\rm (d)}}}}
\def\somee{{\Scale[1.5]{\text{\rm (e)}}}}
\def\Za{{\Scale[4.0]{E_{u\OL{v}}}}}
\def\Zb{{\Scale[4.0]{E_{\OL{u22}\OL{v}}}}}
\def\Zc{{\Scale[4.0]{E_{\OL{u}{v}}}}}
\def\Rec{{\Scale[6.0]{\rho}}}
\def\tq#1{{\Scale[1.1]{#1}}}
\def\Rq#1{{\Scale[0.8]{#1}}}
\def\Qq#1{{\Scale[0.9]{#1}}}
\def\Tq#1{{\Scale[1.0]{#1}}}
\def\Aq#1{{\Scale[2.0]{#1}}}
\def\tarm#1{{\Scale[1.6]{\text{\rm (#1)}}}}
\begin{figure}[ht!]
\centering
\scalebox{0.34}{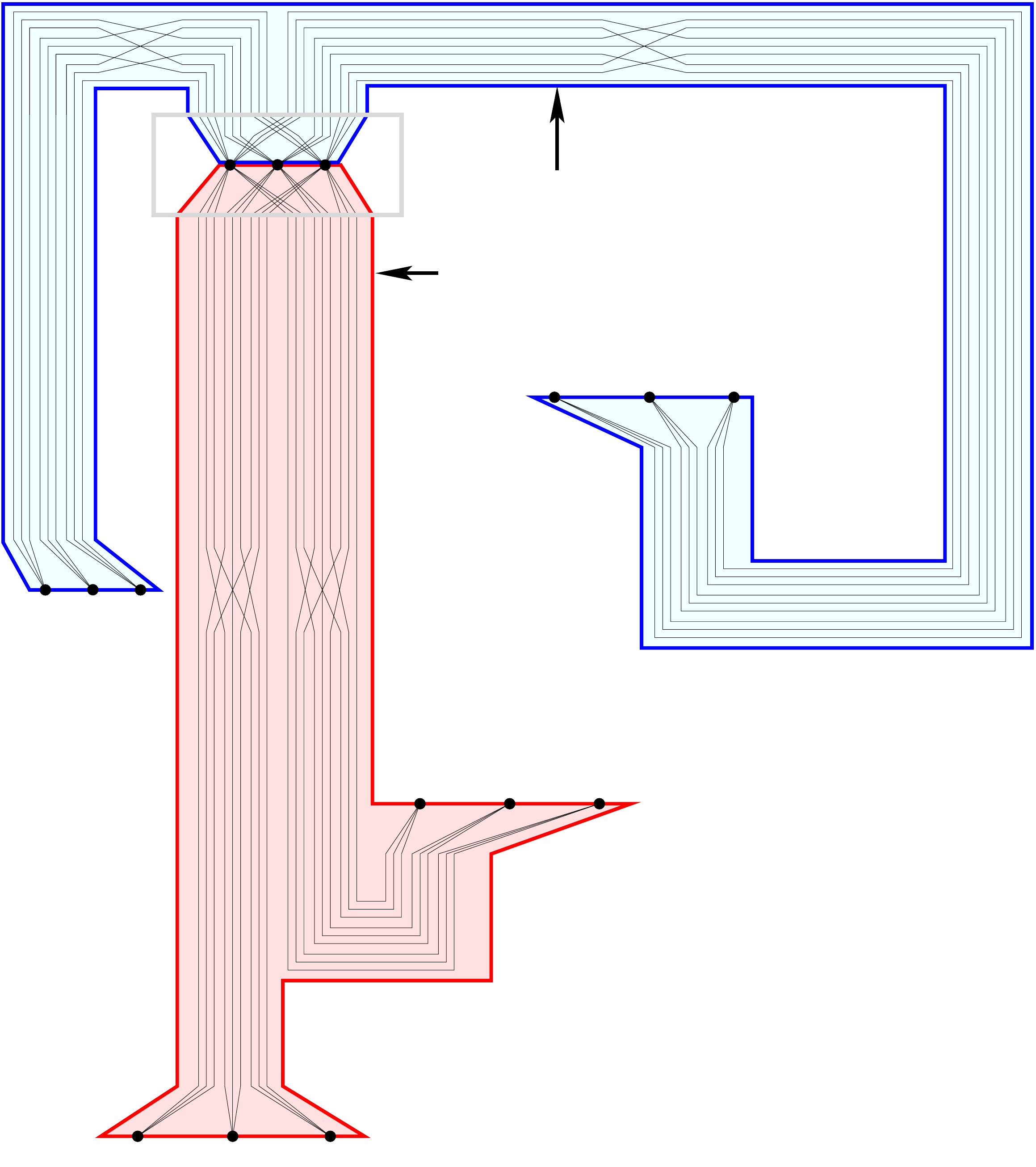}
\caption{This is $L[\Edges{\bom{1}}]$, the subdrawing of the drawing $L$ in Figure~\ref{fig:180} induced by the edges incident with partite class $\bom{1}$. A key feature is that $L[\Edges{\bom{1}}]$ is the crossing-disjoint union of a $\cycper{\bom{1}\cdot \bom{3} \cdot \Rev{\bom{2}} }$-{\outerd} (the blue part) and a $\cycper{\Rev{\bom{1}}\cdot \Rev{\bom{4}}\cdot \bom{5} }$-{\outerd} (the red part).}
\label{fig:225}
\end{figure}

The key property illustrated in Figure~\ref{fig:225} is that $L[\Edges{\bom{1}}]$ is the union of two \outer drawings, the ``blue'' \outer drawing $L[\edges{\bom{1}}{(\bom{2}\cup\bom{3})}]$ (with blue bounding curve $\phi_1$), and the ``red'' \outer drawing $L[\edges{\bom{1}}{(\bom{4}\cup\bom{5})}]$ (with red bounding curve $\chi_1$). 

We note the following key features:

\begin{enumerate}

\item[(A)] \blueit{the blue \outer drawing is a $\cycper{\bom{1}\cdot \bom{3} \cdot \Rev{\bom{2}}}$-{\outerd}, as the vertices in $\bom{1}\cup \bom{2} \cup \bom{3}$ appear in $\phi_1$ in the cyclic order $1(1),1(2),1(3),\, 3(1),3(2),3(3),\, 2(3),2(2),2(1) = \bom{1} \cdot \bom{3} \cdot \Rev{\bom{2}}$;}

\item[(B)]\blueit{the red \outer drawing is a $\cycper{\Rev{\bom{1}} \cdot \Rev{\bom{4}} \cdot \bom{5} }$-\outerd, as the vertices in $\bom{1}\cup\bom {4}\cup \bom{5}$ appear in $\chi_1$ in the cyclic order $1(3),1(2),1(1),\, 4(3),4(2),4(1),\, 5(1),5(2),5(3) = \Rev{\bom{1}}\cdot \Rev{\bom{4}}\cdot \bom{5}$; and}

\item[(C)] the blue \outer drawing and the red \outer drawing are {\em crossing-disjoint}, that is, no edge in the blue drawing crosses an edge in the red drawing.

\end{enumerate}

Thus, the structure of $L[\Edges{\bom{1}}]$ admits a very simple description:

\begin{enumerate}

\item[(L1)] $L[\Edges{\bom{1}}]$ is the crossing-disjoint union of a $\cycper{\bom{1}\cdot \bom{3} \cdot \Rev{\bom{2}} }$-{\outerd} and a $\cycper{\Rev{\bom{1}}\cdot \Rev{\bom{4}}\cdot \bom{5} }$-{\outerd}.

\end{enumerate}

It is easy to verify from Figure~\ref{fig:180} that $L[\Edges{\bom{2}}]$, $L[\Edges{\bom{3}}]$, $L[\Edges{\bom{4}}]$, and $L[\Edges{\bom{5}}]$ also admit easy descriptions:

\begin{enumerate}

\item[(L2)] $L[\Edges{\bom{2}}]$ is the crossing-disjoint union of an empty {\outerd} and a $\cycper{ \Rev{\bom{2}} \cdot \bom{3} \cdot \bom{1} \cdot \Rev{\bom{5}} \cdot \bom{4}  }$-\outerd. \blueit{(We may of course simply say that $L[\Edges{\bom{2}}]$ is a $\cycper{ \Rev{\bom{2}} \cdot \bom{3} \cdot \bom{1} \cdot \Rev{\bom{5}} \cdot \bom{4}  }$-\outerd, but the given description is meant to match the descriptions in (L1), (L3), (L4), and (L5)).}

\item[(L3)] $L[\Edges{\bom{3}}]$ is the crossing-disjoint union of a $\cycper{\bom{3}\cdot \bom{1} \cdot \Rev{\bom{2}}}$-{\outerd} and a $\cycper{ \Rev{\bom{3}}\cdot \bom{5} \cdot \Rev{\bom{4}}}$-\outerd.

\item[(L4)] $L[\Edges{\bom{4}}]$ is the crossing-disjoint union of a $\cycper{\bom{4} \cdot \Rev{\bom{2}}}$-{\outerd} and a $\cycper{ \Rev{\bom{4}}\cdot \Rev{\bom{3}} \cdot \bom{5} \Rev{\bom{1}} }$-\outerd.

\item[(L5)] $L[\Edges{\bom{5}}]$ is the crossing-disjoint union of a 
$\cycper{{\bom{5}}{\cdot} {\Rev{\bom{3}}} {\cdot}{\Rev{\bom{1}}}{\cdot}{\Rev{\bom{4}}}}$-{\outerd} and a ${\cycper{ \Rev{\bom{5}}\cdot \Rev{\bom{2}}}}$-\outerd.

\end{enumerate}

In short, for each $i\in[5]$ we have that $L[\Edges{\bom{i}}]$ is the crossing-disjoint union of two \outer drawings. The bounding order of one of these \outer drawings contains $\bom{i}$, and the bounding order of the other \outer drawing contains $\Rev{\bom{i}}$. Moreover, for each $j\in [5]\setminus \{i\}$ either $\bom{j}$ or $\Rev{\bom{j}}$ appears in (exactly) one of these two bounding orders. 

As we shall see in Section~\ref{sub:canonical}, if we replace $[5]$ with $[m]$ in the previous paragraph we obtain the distinguishing features of {\em every} canonical drawing of a complete multipartite graph.

\subsubsection{Encoding the canonical drawing in Figure~\ref{fig:180}}\label{subsub:encoding}

Let us now describe a way to encode all the relevant information in the drawing in Figure~\ref{fig:180}, which is the way in which we shall encode all canonical drawings.

We start by noting that, informally speaking, for each pair of partite classes $\bom{i},\bom{j}$ there are two ways in which the edges in $\Edges{\bom{i}}{\bom{j}}$ can ``arrive'' as we follow them from $\bom{j}$ to $\bom{i}$. For instance, in Figure~\ref{fig:225} the edges from classes $\bom{2}$ and $\bom{3}$ arrive to the vertices in $\bom{1}$ ``from above'' (they cross the top side of the gray rectangle just before they arrive to the vertices in $\bom{1}$), and the edges from $\bom{4}$ and $\bom{5}$ arrive to $\bom{1}$ ``from below'' (they cross the bottom side of the gray rectangle just before they arrive to the vertices in $\bom{1}$). We also refer the reader to the top left part of Figure~\ref{fig:195}.

We capture this information with a {\em sign function} $\sigma$ by letting $\sigma(2,1)=\sigma(3,1)=1$ and letting $\sigma(4,1)=\sigma(5,1)=-1$. 

In general, keeping the tone informal, $\sigma(j,i)$ is $1$ if the edges coming from $\bom{j}$ arrive to $\bom{i}$ from above, and $\sigma(j,i)$ is $-1$ if the edges coming from $\bom{j}$ arrive to $\bom{i}$ from below. For the drawing in Figure~\ref{fig:180} we capture this basic information in Figure~\ref{fig:195}, illustrating how the edges from each class $\bom{j}$ arrive to each class $\bom{i}$. From this figure we obtain that $\sigma(1,2)=\sigma(3,2)=\sigma(4,2)=\sigma(5,2)=-1$, that $\sigma(1,3)=\sigma(2,3)=1$ and $\sigma(4,3)=\sigma(5,3)=-1$, that $\sigma(2,4)=1$ and $\sigma(1,4)=\sigma(3,4)=\sigma(5,4)=-1$, and that $\sigma(1,5)=\sigma(3,5)=\sigma(4,5)=1$ and $\sigma(2,5)=-1$.

\def\inca{{\Scale[3]{\text{\rm (a)}}}}
\def\incb{{\Scale[3]{\text{\rm (b)}}}}
\def\te#1{{\Scale[4.2]{#1}}}
\def\taf#1{{\Scale[12]{#1}}}
\def\tf#1{{\Scale[8]{#1}}}
\def\tq#1{{\Scale[0.7]{#1}}}
\def\tg#1{{\Scale[5.5]{#1}}}
\def\otg#1{{\Scale[6.5]{#1}}}
\def\somea{{\Scale[1.5]{\text{\rm (a)}}}}
\def\someb{{\Scale[1.5]{\text{\rm (b)}}}}
\def\somec{{\Scale[1.5]{\text{\rm (c)}}}}
\def\somed{{\Scale[1.5]{\text{\rm (d)}}}}
\def\somee{{\Scale[1.5]{\text{\rm (e)}}}}
\def\Za{{\Scale[4.0]{E_{u\OL{v}}}}}
\def\Zb{{\Scale[4.0]{E_{\OL{u22}\OL{v}}}}}
\def\Zc{{\Scale[4.0]{E_{\OL{u}{v}}}}}
\def\Rec{{\Scale[6.0]{\rho}}}
\def\tq#1{{\Scale[1.1]{#1}}}
\def\Rq#1{{\Scale[0.8]{#1}}}
\def\Qq#1{{\Scale[0.9]{#1}}}
\def\Tq#1{{\Scale[1.0]{#1}}}
\def\Aq#1{{\Scale[2.0]{#1}}}
\def\tarm#1{{\Scale[1.6]{\text{\rm (#1)}}}}
\def\tq#1{{\Scale[0.8]{#1}}}
\def\WW#1{{\Scale[1.1]{\bom{#1}}}}
\begin{figure}[ht!]
\centering
\scalebox{0.5}{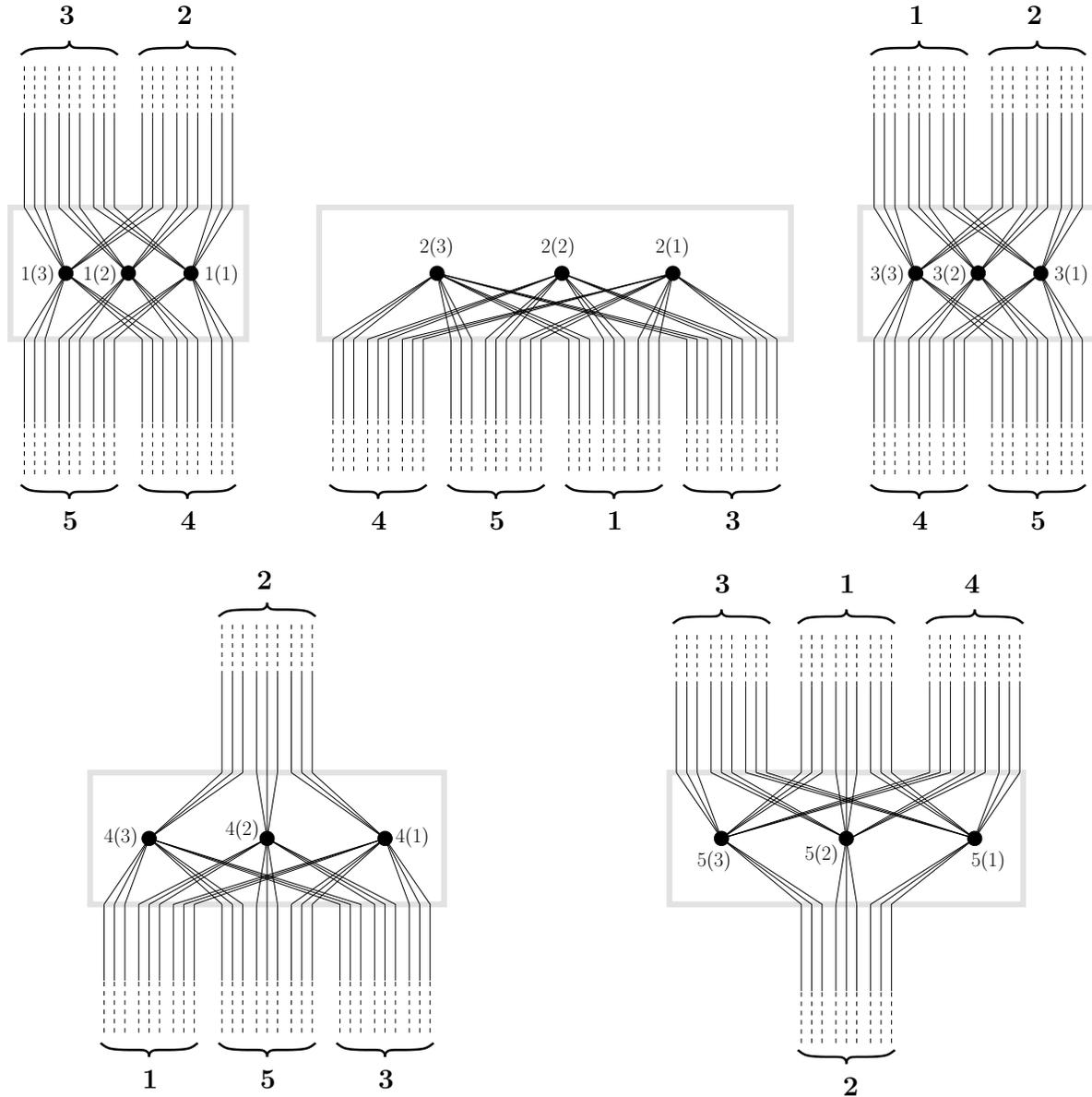}
\caption{We illustrate, for each pair $\bom{i},\bom{j}$ of partite classes, whether the edges coming from class $\bom{j}$ arrive to class $\bom{i}$ from above or from below, in the drawing in Figure~\ref{fig:180}.}
\label{fig:195}
\end{figure}

We now note that for each $i\in[5]$ the function $\sigma$ naturally partitions $[5]\setminus\{i\}$ into two sets:
\begin{align*}
\sigma^+(i)&=\{j\in[5]\setminus\{i\} \, \bigl| \, \sigma(j,i)=1\},\\
\sigma_-(i)&=\{j\in[5]\setminus\{i\} \, \bigl| \, \sigma(j,i)=-1\}.
\end{align*}

Loosely speaking, $j$ is in $\sigma^+(i)$ if the edges from $\bom{j}$ arrive to $\bom{i}$ ``from above'', and $j$ is in $\sigma_-(i)$ if these edges arrive ``from below''.

We note that having $\sigma^+(i)$ and $\sigma_-(i)$ for each $i\in[5]$ we can recover the entire function $\sigma$, as
\[
\sigma(j,i) = 
\begin{cases}
\phantom{-}1, & \textsl{if } j\in \sigma^+(i), \\
-1, & \textsl{if } j\in \sigma_-(i).
\end{cases}
\]


In our running example from Figure~\ref{fig:180} we have the following:

\begin{enumerate}

\item[(i)] $\sigma^+(1)=\{2,3\}$ and $\sigma_-(1)=\{4,5\}$;

\item[(ii)] $\sigma^+(2)=\emptyset$ and $\sigma_-(2)=\{1,3,4,5\}$;

\item[(iii)] $\sigma^+(3)=\{1,2\}$ and $\sigma_-(3)=\{4,5\}$;

\item[(iv)] $\sigma^+(4)=\{2\}$ and $\sigma_-(4)=\{1,3,5\}$; and

\item[(v)] $\sigma^+(5)=\{1,3,4\}$ and $\sigma_-(5)=\{2\}$.

\end{enumerate}

Now having $\sigma^+(i)$ and $\sigma_-(i)$ for each $i\in[5]$ is not enough to recreate all the information in Figure~\ref{fig:180}. Take again for instance the case $i=1$. Besides knowing that classes $\bom{2}$ and $\bom{3}$ ``arrive at class $\bom{1}$ from above'', we need to specify in which order they do so. In order to have a convention we use the gray rectangle that encloses $\bom{1}$ in that figure (see also Figures~\ref{fig:225} and~\ref{fig:195} for better views). 

Regarding the classes that arrive from above, we record the order in which they hit the gray rectangle as we traverse the top side of the rectangle {\em from left to right}. In this case, we encounter the edges coming from $\bom{3}$ first, and then the edges coming from $\bom{2}$. Thus the set $\sigma^+(1)=\{2,3\}$ yields the permutation $\perm{3,2}$. We use $1^+$ to denote this permutation of $\sigma^+(1)$.

Regarding the classes that arrive from below, we record the order in which they hit the gray rectangle as we traverse the bottom side of the rectangle {\em from right to left}. In this case, we encounter the edges coming from $\bom{4}$ first, and then the edges coming from $\bom{5}$. Thus the set $\sigma_-(1)=\{4,5\}$ yields the permutation $\perm{4,5}$. We use $1_-$ to denote this permutation of $\sigma_-(1)$.

Thus in our running example we have $1^+=\perm{3,2}$ and $1_-=\perm{4,5}$. Performing the same procedure for $i=2,3,4,5$ we obtain the following refined information from the drawing in Figure~\ref{fig:180} (again, Figure~\ref{fig:195} is very helpful in this task):

\begin{enumerate}

\item[(1)] $1^+=\perm{3,2}$ and $1_-=\perm{4,5}$;

\item[(2)] $2^+=\perm{}$ and $2_-=\perm{3,1,5,4}$;

\item[(3)] $3^+=\perm{1,2}$ and $3_-=\perm{5,4}$;

\item[(4)] $4^+=\perm{2}$ and $4_-=\perm{3,5,1}$;

\item[(5)] $5^+=\perm{3,1,4}$ and $5_-=\perm{2}$.

\end{enumerate}

\blueit{In Section~\ref{sub:canonical} we shall call the collection $\bigl((1^+,1_-),(2^+,2_-),(3^+,3_-),(4^+,4_-),(5^+,5_-)\bigr)$ a {\em template} for this particular canonical drawing. We close this section by showing that this collection of permutations allows us to reconstruct the information given in (L1)--(L5).}



\blueit{For each $i\in[5]$ let $\perm{i^1,\ldots,i^{|\sigma^+(i)|}}$ be the permutation $i^+$. From this permutation we create a cyclic permutation starting with $\bom{i}$ and then including $\bom{i^1},\ldots,\bom{i^{|\sigma^+(i)|}}$ in this order, but placing the sign $\sigma(i,i^j)$ in front of $\bom{i^j}$ for $j=1,\ldots,|\sigma^+(i)|$. That is, from $i^+=\perm{i^1,\ldots,i^{|\sigma^+(i)|}}$ we get the cyclic permutation $\cycper{\bom{i}\cdot \sigma(i,i^1)\, \bom{i^1}\cdot \cdots \cdot \sigma(i,i^{|\sigma^+(i)|})\, \bom{i^{|\sigma^+(i)|}}}$.}

\blueit{For instance, the permutation $1^+=\perm{3,2}$ induces the cyclic permutation $\cycper{\bom{1}\cdot \sigma(1,3)\,\bom{3} \cdot \sigma(1,2)\, \bom{2}}$. Now $\sigma(1,3)=1$ (since $1\in 3^+$) and $\sigma(1,2)=-1$ (since $1\in 2_{-}$), and so this cyclic permutation is $\cycper{\bom{1}\cdot \bom{3} \cdot \Rev{\bom{2}}}$, as in (L1).}

\blueit{Similarly, let $\perm{i_1,\ldots,i_{|\sigma_{-}(i)|}}$ be the permutation $i_{-}$. From this permutation we create the cyclic permutation that starts with $\Rev{\bom{i}}$ and includes $\bom{i_1},\ldots,\bom{i_{|\sigma_{-}(i)|}}$ in this order, but placing the sign $\sigma(i,i_j)$ in front of $\bom{i_j}$ for $j=1,\ldots,|\sigma_{-}(i)|$. That is, from $i_{-}=\perm{i_1,\ldots,i_{|\sigma_{-}(i)|}}$ we get the cyclic permutation $\cycper{\Rev{\bom{i}}\cdot \sigma(i,i_1)\, \bom{i_1}\cdot \cdots \cdot \sigma(i,i_{|\sigma_{-}(i)|})\, \bom{i_{|\sigma_{-}(i)|}}}$.}

\blueit{For instance, $1_{-}=\perm{4,5}$ induces the cyclic permutation $\cycper{\Rev{\bom{1}}\cdot \sigma(1,4)\,\bom{4} \cdot \sigma(1,5)\, \bom{5}}$. Now $\sigma(1,4)=-1$ (since $1\in 4_{-}$) and $\sigma(1,5)=1$ (since $1\in 5^{+}$), and so this cyclic permutation is $\cycper{\Rev{\bom{1}}\cdot \Rev{\bom{4}} \cdot \bom{5}}$, as in (L1).}

\blueit{In totally analogous manner it is easily verified that $2_-=\perm{3,1,5,4}$ induces the cyclic permutation $\cycper{ \Rev{\bom{2}} \cdot \bom{3} \cdot \bom{1} \cdot \Rev{\bom{5}} \cdot \bom{4}  }$ in (L2) (the empty permutation $2^+=\perm{}$ does not induce any cyclic permutation); the permutations $3^+=\perm{1,2}$ and $3_-=\perm{5,4}$ induce the cyclic permutations $\cycper{\bom{3}\cdot \bom{1} \cdot \Rev{\bom{2}}}$ and $\cycper{ \Rev{\bom{3}}\cdot \bom{5} \cdot \Rev{\bom{4}}}$ in (L3), respectively; the permutations $4^+=\perm{2}$ and $4_-=\perm{3,5,1}$ induce the cyclic permutations $\cycper{\bom{4} \cdot \Rev{\bom{2}}}$ and $\cycper{ \Rev{\bom{4}}\cdot \Rev{\bom{3}} \cdot \bom{5} \Rev{\bom{1}} }$ in (L4), respectively; and $5^+=\perm{3,1,4}$ and $5_-=\perm{2}$ induce the cyclic permutations $\cycper{{\bom{5}}{\cdot} {\Rev{\bom{3}}} {\cdot}{\Rev{\bom{1}}}{\cdot}{\Rev{\bom{4}}}}$ and $\cycper{ \Rev{\bom{5}}\cdot \Rev{\bom{2}}}$, respectively.}

\blueit{Therefore, as claimed, with the {\em template} $\bigl((1^+,1_-),(2^+,2_-),(3^+,3_-),(4^+,4_-),(5^+,5_-)\bigr)$ we can reproduce Properties (L1)--(L5), which in turn provide a description of the drawing $L$ in Figure~\ref{fig:180}.}


\subsection{Canonical drawings of $K_n^m$}\label{sub:canonical}

As we did in Section~\ref{subsub:encoding} for the drawing in Figure~\ref{fig:180}, every canonical drawing is encoded by giving for each $i\in[m]$ two (linear) permutations $i^+$ and $i_-$ such that each $j\in[m]\setminus\{i\}$ appears once in exactly one of $i^+$ and $i_-$. In other words, $i^+\cdot i_-$ is a permutation of $[m]\setminus\{i\}$. This is captured in a single entity under the notion of a template, which in turn involves the concept of a sign function.

\begin{definition}[Sign functions]
{\sl Let $m$ be a positive integer. A {\em sign function over $[m]$} (or simply a {\em sign function}, if $m$ is clear in the context) is a function $\sigma$ that assigns to each ordered pair $(j,i)$ of distint $i,j\in[m]$ an integer $\sigma(j,i)$ in $\{-1,1\}$. For each $j\in[m]$ we let $\sigma^+(i):=\{j\in[m]\setminus\{i\}\, | \, \sigma(j,i)=1\}$ and $\sigma_-(i):=\{j\in[m]\setminus\{i\}\, | \, \sigma(j,i)=-1\}$.}
\end{definition}
 
\begin{definition}[Templates]
{\sl Let $m$ be a positive integer, and let $\sigma$ be a sign function over $[m]$. For each $i\in[m]$ let $i^+$ (respectively, $i_-$) be a permutation of $\sigma^+(i)$ (respectively, $\sigma_-(i)$). We say that $\Gamma=\bigl((1^+,1_-),\ldots,(m^+,m_-)\bigr)$ is a {\em template}, and that $\sigma$ is the sign function of the template $\Gamma$.} 
\end{definition}

We are finally ready to define what is a canonical drawing of $K_n^m$.

\begin{definition}[Canonical drawings of $K_n^m$] 
{\sl Let $C$ be a drawing of $K_n^m$ with the partite classes labelled $\bom{1},\ldots,\bom{m}$. We say that $C$ is {\em canonical} if there is a template $\Gamma=\bigl((1^+,1_{-}),(2^+,2_{-}),\ldots,(m^+,$ $m_{-})\bigr)$ with the following property. Let $\sigma$ be the  sign function of $\Gamma$. Then the following holds for each $i\in[m]$.}

{\sl Let $\perm{i^1,\ldots,i^{\sigma^+(i)}}$ be the permutation $i^+$. Then:}

\begin{enumerate}

\item[(C1)] {\sl $C[\edges{\bom{i}}{(\bom{i^1}\cup\cdots\cup\bom{i^{|\sigma^+(i)|}})}]$ is a $\cycper{\bom{i} \cdot \sigma(i,i^1)\, \bom{i^1}\cdot\,\,\,\cdots\,\,\,\cdot \sigma(i,i^{|\sigma^+(i)|})\, \bom{i^{|\sigma^+(i)|}}}$-\outerd.}

\end{enumerate}

{\sl Let $\perm{i_1,\ldots,i_{\sigma_-(i)}}$ be the permutation $i_{-}$. Then:}

\begin{enumerate}

\item[(C2)] {\sl $C[\edges{\bom{i}}{({\bom{i_1}}\cup\cdots\cup\bom{i_{|\sigma_{-}(i)|}})}]$ is a $\cycper{\Rev{\bom{i}} \cdot \sigma(i,i_1)\, \bom{i_1}\cdot\,\,\,\cdots\,\,\,\cdot \sigma(i,i_{|\sigma_{-}(i)|})\, \bom{i_{|\sigma_{-}(i)|}}}$-\outerd.}

\end{enumerate}

{\sl Finally,}

\begin{enumerate}

\item[(C3)] {\sl no edge in $C[\edges{\bom{i}}{(\bom{i^1}\cup\cdots\cup\bom{i^{|\sigma^+(i)|}})}]$ crosses an edge in $C[\edges{\bom{i}}{({\bom{i_1}}\cup\cdots\cup\bom{i_{|\sigma_{-}(i)|}})}]$}.

\end{enumerate}

We note that since $\sigma^+(i)$ is the set $\{i^1,\ldots,i^{|\sigma^+(i)|}\}$ and $\sigma_-(i)$ is the set $\{i_1,\ldots,i_{|\sigma_{-}(i)|}\}$, (C3) can be equivalently paraphrased as follows:

\begin{enumerate}

\item[(C3$'$)] {\sl for all distinct $i,j,k\in [m]$, if $j\in\sigma^+(i)$ and $k\in\sigma_-(i)$, then no edge in $C[\Edges{\bom{i}}{\bom{j}}]$ crosses an edge in $C[\Edges{\bom{i}}{\bom{k}}]$.} 

\end{enumerate}

\end{definition}

Before we state our main result let us capture an important feature of templates.

\begin{lemma}\label{lem:templates1}
A template determines a canonical drawing up to weak isomorphism. That is, if two canonical drawings have the same template then they are weakly isomorphic.
\end{lemma}

We defer the proof of this lemma to a later section in this paper (namely Section~\ref{sec:prooftemplates1}), so we can now proceed to state our main result.


\subsection{Our main result}\label{sub:main}

As we mentioned above, our main result is that canonical drawings are precisely the unavoidable drawings of complete multipartite graphs:

\begin{theorem}[The unavoidable drawings of $K_n^m$]\label{thm:main}
Let $m\ge 2$, $N > n\ge 3$ be  integers. If $N$ is sufficiently large compared to $n$, then every drawing of $K_N^m$ contains a drawing weakly isomorphic to a canonical drawing of $K_n^m$.
\end{theorem}

Before we proceed to the proof of Theorem~\ref{thm:main} let us raise an important issue. The theorem identifies the unavoidable drawings of $K_n^m$, in the sense that it establishes that these are precisely the canonical drawings. Since each canonical drawing gets determined by a template, for a full identification one needs to address a crucial \blueit{question}: when is a template the template of a canonical drawing? \blueit{We close this section with a discussion on this question, which we fully answer in Section~\ref{sec:realizable}.}

\subsection{Realizable and non-realizable templates}\label{sub:realizable}

Let us recall that templates are defined in purely combinatorial terms. That is, if for each $i\in [m]$ we choose any permutations $i^+$ and $i_{-}$ such that $i^+ \cdot i_{-}$ is a permutation of $[m]\setminus\{i\}$, then a result we obtain a (combinatorially) valid template. On the other hand, there is no reason for such a template to be the template of a drawing of $K_n^m$.

\blueit{Let us say that a template is {\em realizable} if it is the template of some drawing of a complete multipartite graph $K_n^m$. The previous discussion revolves around the possible existence of non-realizable templates. Our next statement gives an explicit example of such a template.}

\begin{observation}\label{obs:nonrealizable}
\textsl{\blueit{The template $\Gamma=\bigl((1^+,1_{-}),(2^+,2_{-}),(3^+,3_{-}),(4^+,4_{-})\bigr)$, where \[1^+=\perm{2,3,4}, 1_{-}=\perm{ }, 2^+=\perm{3,4,1},2_{-}=\perm{ }, 3^+=\perm{4,1,2},3_{-}=\perm{ }, 4^+=\perm{1,3,2} \textrm{ and } 4_{-}=\perm{ }\] is not realizable. That is, it is not the template of any drawing of $K_n^4$ for any positive integer $n$.}}
\end{observation}

\begin{proof}
For a contradiction suppose that $\Gamma$ is the template of a drawing of $K_n^4$ for some integer $n$. In particular, this implies that $\Gamma$ is the template of a drawing $C$ of $K_2^4$.

For each $i\in [4]$ let partite class $\bom{i}$ be the permutation $\perm{i(1),i(2)}$. Note that for this template $\Gamma$ the sign function $\sigma$ is quite simple: since $1_-,2_-,3_-$, and $4_-$ are all empty permutations, it follows that $\sigma(j,i)=+1$ for each pair of distinct $i,j\in[4]$.

\blueit{We derive a contradiction by focusing on the restriction $C'$ of $C$ to the $K_4$ induced by the four vertices $1(1),2(1), 3(1)$, and $4(1)$ (we choose $i(1)$ for $i=1,2,3,4$ for definiteness, but we might have chosen any vertex in each partite class). We shall show that the assumption that $\Gamma$ is the template of $C$ implies that the rotation system of $C'$ is not realizable, that is, that this cannot be the rotation system of a simple drawing of $K_4$.}

\blueit{To see this let us start by noting that Property (C1) for $i=1$ means that $C[\edges{\bom{1}}{(\bom{2}\cup \bom{3}\cup \bom{4})} ]$ is a $\cycper{\bom{1}\cdot \bom{2} \cdot \bom{3}\cdot \bom{4}}$-\outerd, as illustrated in Figure~\ref{fig:705}. In particular, the rotation at vertex $1(1)$ in $C$ is $\cycper{2(1),2(2),3(1),3(2),4(1),4(2)}$. Therefore the rotation at vertex $1(1)$ in $C'$ is $\cycper{2(1),3(1),4(1)}$.}

\begin{figure}[ht!]
\def\ta#1{{\Scale[2.2]{#1}}}
\def\tb#1{{\Scale[0.8]{#1}}}
\centering
\scalebox{0.35}{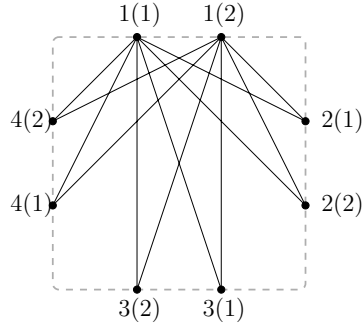}
\caption{{Illustration of the proof of Observation~\ref{obs:nonrealizable}}}
\label{fig:705}
\end{figure}

\blueit{Analogous arguments applying Property (C1) for $i=2,3$, and $4$ imply that in $C'$ (a) the rotation at vertex $2(1)$ is $\cycper{3(1),4(1),1(1)}$; (b) the rotation at $3(1)$ is $\cycper{4(1),1(1),2(1)}$; and the rotation at $4(1)$ is $\cycper{1(1),3(1),2(1)}$. Thus the rotation system of $C'$ is $\{\cycper{2(1),3(1),4(1)}, \cycper{3(1),4(1),1(1)}$, $\cycper{4(1),1(1),2(1)}, \cycper{1(1),3(1),2(1)}\}$. This yields the required contradiction, as it is easy to see that no simple drawing of $K_4$ with its vertices labelled $1(1),2(1),3(1),4(1)$ has this rotation system.}
\end{proof}

\blueit{In Section~\ref{sec:realizable} we give necessary and sufficient conditions for a template to be realizable.}



\section{Reducing Theorem~\ref{thm:main} to two propositions}\label{sec:proofmain}

\blueit{In this section we put forward three statements (a corollary of Theorem~\ref{thm:negami2} and two propositions) and show that they imply Theorem~\ref{thm:main}. The proofs of the propositions will be deferred to later sections.}

We start with an easy consequence of Theorem~\ref{thm:negami2}. Loosely speaking, if $N \ge N_1$ are integers and $N$ is sufficiently large compared to $N_1$, then every drawing of $K_N^m$ contains a subdrawing of $K_{N_1}^m$ that is ``pairwise \outer'', that is, a drawing in which the restriction to any two partite classes is as illustrated in (one of the cases in) Figure~\ref{fig:445}. 

Formally, let $J$ be a drawing of $K_n^m$ with the partite classes labelled $\bom{1},\ldots,\bom{m}$. We say that $J$ is {\em pairwise \outer} if for all distinct $i,j\in[m]$ we have that $J[\Edges{\bom{i}}{\bom{j}}]$ is either a $\cycper{\bom{i}\,\cdot\,\bom{j}}$-drawing, or a $\cycper{\Rev{\bom{i}}\,\cdot\,\bom{j}}$-drawing, or a $\cycper{\bom{i}\,\cdot\,\Rev{\bom{j}}}$-drawing, or a $\cycper{\Rev{\bom{i}}\,\cdot\,\Rev{\bom{j}}}$-drawing.

An equivalent, more compact way to say this is that $J$ is \textit{pairwise \outer} if there is a sign function $\sigma$ over $[m]$ such that $J[\Edges{\bom{i}}{\bom{j}}]$ is a $\cycper{\sigma(j,i)\,\bom{i} \,\cdot\, \sigma(i,j)\,\bom{j}}$-drawing, for all distinct $i,j\in[m]$. We say that $\sigma$ is {\em the pairwise sign function} of $J$. 

An iterative application of Theorem~\ref{thm:negami2} (over each pair of distinct integers $i,j\in[m]$) yields the following.

\begin{corollary}[Follows from Theorem~\ref{thm:negami2}]\label{cor:negami}
Let $m\ge 2$, $Q > q \ge 1$ be integers. Let $J$ be a drawing of $K_Q^{m}$ with the partite classes labelled $\BO{1},\ldots,\BO{m}$. If $Q$ is sufficiently large compared to $q$, then there exist $\bom{1}\preceq\BO{1},\ldots,\bom{m}\preceq\BO{m}$ with  $|\bom{1}|=\cdots=|\bom{m}|=q$, such that $J[\bom{1}\cup\cdots\cup\bom{m}]$ is a pairwise \outer drawing.
\end{corollary}

We now state two propositions that are at the heart of the proof of Theorem~\ref{thm:main}. The proofs of these propositions are deferred to Sections~\ref{sec:core} and~\ref{sec:core2}, respectively. 

\begin{proposition}\label{pro:core}
Let $Q > q \ge 1$ be integers. Let $\oA,\oB_1,\ldots,\oB_r$ be pairwise disjoint permutations of vertices in a graph $G$, with $|\oA|=|\oB_1|=\cdots=|\oB_r|=Q$, such that every vertex in $\oA$ is adjacent to every vertex in $\oB_1\cup\cdots\cup \oB_r$. Suppose that $J$ is a drawing of $G$ such that $J[\edges{\oA}{\oB_i}]$ is a $\cycper{\oA \cdot \oB_i}$-{\outerd} for all $1\le i \le r$. If $Q$ is sufficiently large compared to $q$, then there exist subpermutations $A\preceq \oA$ and $B_i\preceq \oB_i$ for $i=1,\ldots,r$, with $|A|=|B_1|=\cdots=|B_{r}|=q$ and a permutation $\pi=\perm{\pi(1),\ldots,\pi(r)}$ of $[r]$ such that $J[\edges{A}{(B_1\cup \,\,\cdots\,\,\cup B_r)}]$ is a $\cycper{A\cdot B_{\pi(1)}\cdot \,\,\cdots\,\,\cdot B_{\pi(r)}}$-\outerd.
\end{proposition}

\begin{proposition}\label{pro:core2}
Let $Q > q \ge 1$ be integers. Let $\oA, \oB$, and $\oC$ be pairwise disjoint permutations of vertices in a graph $G$, with $|\oA|=|\oB|=|\oC|=Q$, such that every vertex in $\oA$ is adjacent to every vertex in $\oB\cup\oC$. Let $J$ be a drawing of $G$ such that $J[\Edges{\oA}{\oB}]$ is a $\cycper{\oA\cdot\oB}$-drawing and $J[\Edges{\oA}{\oC}]$ is a $\cycper{\Rev{\oA}\cdot\oC}$-drawing. If $Q$ is sufficiently large compared to $q$, then there exist subpermutations $A\preceq \oA, B\preceq \oB$, and $C\preceq\oC$ with $|A|=|B|=|C|=q$ such that no edge in $J[\Edges{A}{B}]$ crosses an edge in $J[\Edges{A}{C}]$.
\end{proposition}

Theorem~\ref{thm:main} simply claims that every sufficiently large drawing of $K_N^m$ contains a subdrawing that is a canonical drawing of $K_n^m$. Before we prove the theorem let us restate it fully including the definition of a canonical drawing.

\vglue 0.4 cm
\noindent{\bf Theorem~\ref{thm:main} (Equivalent formulation). }{\em Let $m\ge 2$ and $N\ge n>1$ be integers. Let $I$ be a drawing of $K_N^m$ with the partite classes labelled $\BO{1},\ldots,\BO{m}$. If $N$ is sufficiently large compared to $n$, then there exist $\bom{1}\preceq\BO{1},\ldots,\bom{m}\preceq\BO{m}$ with  $|\bom{1}|=\cdots=|\bom{m}|=n$ and a template $\Gamma=\bigl((1^+,1_{-}),(2^+,2_{-}),\ldots,(m^+,$ $m_{-})\bigr)$ with sign function $\sigma$ such that the following holds for each $i\in[m]$.}

{\sl Let $\perm{i^1,\ldots,i^{\sigma^+(i)}}$ be the permutation $i^+$. Then:}

\begin{enumerate}

\item[(C1)] {\sl $I[\edges{\bom{i}}{(\bom{i^1}\cup\cdots\cup\bom{i^{|\sigma^+(i)|}})}]$ is a $\cycper{\bom{i} \cdot \sigma(i,i^1)\, \bom{i^1}\cdot\,\,\,\cdots\,\,\,\cdot \sigma(i,i^{|\sigma^+(i)|})\, \bom{i^{|\sigma^+(i)|}}}$-\outerd.}

\end{enumerate}

{\sl Let $\perm{i_1,\ldots,i_{\sigma_-(i)}}$ be the permutation $i_{-}$. Then:}

\begin{enumerate}

\item[(C2)] {\sl $I[\edges{\bom{i}}{({\bom{i_1}}\cup\cdots\cup\bom{i_{|\sigma_{-}(i)|}})}]$ is a $\cycper{\Rev{\bom{i}} \cdot \sigma(i,i_1)\, \bom{i_1}\cdot\,\,\,\cdots\,\,\,\cdot \sigma(i,i_{|\sigma_{-}(i)|})\, \bom{i_{|\sigma_{-}(i)|}}}$-\outerd.}

\end{enumerate}

{\sl Finally,}

\begin{enumerate}

\item[(C3$'$)] {\sl for all distinct $i,j,k\in [m]$, if $j\in\sigma^+(i)$ and $k\in\sigma_-(i)$, then no edge in $I[\Edges{\bom{i}}{\bom{j}}]$ crosses an edge in $I[\Edges{\bom{i}}{\bom{k}}]$.} 

\end{enumerate}

\begin{proof}[Proof of Theorem~\ref{thm:main}, assuming Propositions~\ref{pro:core} and~\ref{pro:core2}]
Let $m\ge 2$ and $N > n\ge 1$, be integers, and let $I$ be a drawing of $K_N^m$ with partite classes $\BO{1},\ldots,\BO{m}$.

We start by noting that in view of Corollary~\ref{cor:negami}, the freedom to assume that $N$ is arbitrarily large allows us to assume  that $I$ is a pairwise \outer drawing. 

As we argue below, the proof of the theorem easily follows from the next statement.

\vglue 0.3cm
\noindent\textbf{Claim. }\textsl{Let $m\ge 2$ and $Q > q\ge 1$, be integers, and let $J$ be a pairwise \outer drawing of $K_Q^m$ with partite classes $\BO{1},\ldots,\BO{m}$, and pairwise sign function $\sigma$. Then:}
\begin{enumerate}
\item[(I)] \textsl{Let $i$ be a fixed integer in $[m]$. If $Q$ is sufficiently large compared to $q$, then there exist $\bom{1}\preceq\BO{1},\ldots,\bom{m}\preceq\BO{m}$ with  $|\bom{1}|=\cdots=|\bom{m}|=q$, and a permutation $i^+:=\perm{i^1,\ldots,i^{\sigma^+(i)}}$ of $\sigma^+(i)$ such that}
\[
J[\edges{\bom{i}}{(\bom{i^1}\cup\cdots\cup\bom{i^{|\sigma^+(i)|}})}] 
\text{\sl \,\,is a\,\,}
\cycper{\bom{i} \cdot \sigma(i,i^1)\, \bom{i^1}\cdot\,\,\,\cdots\,\,\,\cdot \sigma(i,i^{|\sigma^+(i)|})\, \bom{i^{|\sigma^+(i)|}}}{\text{\sl -drawing}}.
\]

\item[(II)] \textsl{Let $i$ be a fixed integer in $[m]$. If $Q$ is sufficiently large compared to $q$, then there exist $\bom{1}\preceq\BO{1},\ldots,\bom{m}\preceq\BO{m}$ with  $|\bom{1}|=\cdots=|\bom{m}|=q$, and a permutation $i_-:=\perm{i_1,\ldots,i_{\sigma_-(i)}}$ of $\sigma_-(i)$ such that}
\[
J[\edges{\bom{i}}{(\bom{i_1}\cup\cdots\cup\bom{i_{|\sigma_-(i)|}})}] 
\text{\sl \,\,is a\,\,}
\cycper{\Rev{\bom{i}} \cdot \sigma(i,i_1)\, \bom{i_1}\cdot\,\,\,\cdots\,\,\,\cdot \sigma(i,i_{|\sigma_-(i)|})\, \bom{i_{|\sigma_-(i)|}}}{\text{\sl -drawing}}.
\]
\item[(III)] \textsl{Let $i,j,k$ be distinct fixed integers in $[m]$ such that $j\in\sigma^+(i)$ and $k\in\sigma_-(i)$. If $Q$ is sufficiently large compared to $q$, then there exist $\bom{1}\preceq\BO{1},\ldots,\bom{m}\preceq\BO{m}$ with  $|\bom{1}|=\cdots=|\bom{m}|=q$ such that no edge in $J[\Edges{\bom{i}}{\bom{j}}]$ crosses an edge in $J[\Edges{\bom{i}}{\bom{k}}]$.} 
\end{enumerate}

\vglue 0.3cm

To see that this Claim indeed implies the theorem we start by noting that if $N$ is sufficiently large then we can apply (I) to $i=1,\ldots,m$ successively, then (II) to $i=1,\ldots,m$ successively, and then (III) also successively to every triple $i,j,k$ of distinct integers in $[m]$ such that $j\in \sigma^+(i)$ and $k\in\sigma_-(i)$. (After each application (I), (II), or (III) we relabel back the resulting subpermutations $\bom{1},\ldots,\bom{m}$ with $\BO{1},\ldots,\BO{m}$, so that we can apply (I), (II), or (III) seamlessly once again).

As a final result we obtain subpermutations $\bom{1},\ldots,\bom{m}$ of the original partite classes $\BO{1},\ldots,\BO{m}$, with $|\bom{1}|=\cdots=|\bom{m}|$ and $1^+, 1_-, \ldots, m^+, m_-$ (that is, a template $\Gamma=\bigl((1^+,1_-),\ldots,(m^+,m_-)\bigr)$) such that (C1), (C2), and (C3$'$) hold, as claimed in Theorem~\ref{thm:main}.

Thus we complete the proof by showing the Claim.

\begin{proof}[Proof of the Claim]
We start with the proof of (I) by noting that for simplicity it suffices to prove it for the particular case $i=1$, as the proof for an arbitrary value of $i$ is identical. 

Thus it suffices to show that if $Q$ is sufficiently large compared to $q$, then there exist  $\bom{1}\preceq\BO{1},$ $\ldots,\bom{m}\preceq\BO{m}$ with  $|\bom{1}|=\cdots=|\bom{m}|=q$, and a permutation $1^+=\perm{1^1,\ldots,1^{\sigma^+(1)}}$ of $\sigma^+(1)$ such that
\begin{equation}\label{eq:ClaI}
J[\edges{\bom{1}}{(\bom{1^1}\cup\cdots\cup\bom{1^{|\sigma^+(1)|}})}] 
\text{\sl \,\,is a\,\,}
\cycper{\bom{1} \cdot \sigma(1,1^1)\, \bom{1^1}\cdot\,\,\,\cdots\,\,\,\cdot \sigma(1,1^{|\sigma^+(1)|})\, \bom{1^{|\sigma^+(1)|}}}{\text{\sl -drawing}}.
\end{equation}

But this is an easy consequence of Proposition~\ref{pro:core}: simply relabel $\BO{1}$ with $\oA$, and relabel the $|\sigma^+(1)|$ permutations in $\{\sigma(1,i)\,\BO{i}\,\bigl| \, i\in \sigma^+(1)\}$ with $\oB_1,\ldots,\oB_{|\sigma^+(1)|}$ arbitrarily. With these relabellings, the assumption that $J$ is pairwise \outer means that $J[\Edges{\oA}{\oB_i}]$ is a $\cycper{\oA\,\cdot\,\oB_i}$-drawing for all $1 \le i \le |\sigma^+(1)|$. 

Proposition~\ref{pro:core} then implies that if $Q$ is sufficiently large compared to $q$, then there exist subpermutations $A\preceq \oA$ and $B_i\preceq \oB_i$ for $i=1,\ldots,|\sigma^+(1)|$ with $|A|=|B_1|=\cdots=|B_{|\sigma^+(1)|}|$, and a permutation $\pi=\perm{\pi(1),\ldots,\pi(|\sigma^+(1)|)}$ of $\{1,\ldots,|\sigma^+(1)|\}$ such that 
\begin{equation}\label{eq:ClaI2}
J[\edges{A}{(B_1\cup \,\,\cdots\,\,\cup B_{|\sigma^+(1)|})}] \text{\sl \,\,is a\,\,}
\cycper{A\cdot B_{\pi(1)}\cdot \,\,\cdots\,\,\cdot B_{\pi(|\sigma^+(1))|}}\textsl{-\outerd.}
\end{equation}

Now if we relabel $A$ with $\bom{1}$ and $B_{\pi(j)}$ with $\sigma(1,1^j) \, \bom{1^j}$ for $j=1,\ldots,|\sigma^+(1)|$, \eqref{eq:ClaI2} becomes exactly~\eqref{eq:ClaI}, as required. This completes the proof of (I).

The proof of (I) is easily adapted to prove (II).

Turning our attention to (III), let $i,j,k$ be any fixed triple of distinct integers in $[m]$ such that $\sigma(j,i)=1$ and $\sigma(k,i)=-1$. For simplicity, by relabelling the partite classes if necessary we may assume that $i=1, j=2$, and $k=3$, so that $\sigma(2,1)=1$ and $\sigma(3,1)=-1$. 

Our aim is to show that if $Q$ is sufficiently large compared to $q$, then there exist $\bom{1}\preceq\BO{1}, \bom{2}\preceq\BO{2}$, and $\bom{3}\preceq\BO{3}$ with $|\bom{1}|=\cdots=|\bom{m}|=q$ such that no edge in $J[\Edges{\bom{1}}{\bom{2}}]$ crosses an edge in $J[\Edges{\bom{1}}{\bom{3}}]$. {Note that the partite classes $\BO{4},\ldots,\BO{m}$ are irrelevant for this purpose: $\bom{4},\ldots,\bom{m}$ can be chosen to be any subpermutations of size $n$ of $\BO{4},\ldots,\BO{m}$, respectively.}

The assumption that $J$ is a pairwise \outer drawing implies (since $\sigma(2,1)=1$) that $J[\edges{\BO{1}}{\BO{2}}]$ is either a $\cycper{\BO{1}\cdot\BO{2}}$-drawing or a $\cycper{\BO{1}\cdot\Rev{\BO{2}}}$-drawing. Similarly, since $\sigma(3,1)=-1$ we have that $J[\edges{\BO{1}}{\BO{3}}]$ is either a $\cycper{\Rev{\BO{1}}\cdot\BO{3}}$-drawing or a $\cycper{\Rev{\BO{1}}\cdot\Rev{\BO{3}}}$-drawing.

By reversing one or both of the permutations $\BO{2}$ and $\BO{3}$ if necessary, for simplicity we may assume that $J[\edges{\BO{1}}{\BO{2}}]$ is a $\cycper{\BO{1}\cdot\BO{2}}$-drawing and $J[\edges{\BO{1}}{\BO{3}}]$ is a $\cycper{\Rev{\BO{1}}\cdot\BO{3}}$-drawing.

We now relabel $\BO{1}$ with $\oA$, $\BO{2}$ with $\oB$, and $\BO{3}$ with $\oC$. Thus $J[\Edges{\oA}{\oB}]$ is a $\cycper{\oA\cdot\oB}$-drawing and $J[\Edges{\oA}{\oC}]$ is a $\cycper{\Rev{\oA}\cdot\oC}$-drawing. By Proposition~\ref{pro:core2}, if $Q$ is sufficiently large compared to $q$ then there exist subpermutations $A\preceq \oA, B\preceq \oB$, and $C\preceq\oC$ with $|A|=|B|=|C|={q}$ such that no edge in $J[\Edges{A}{B}]$ crosses an edge in $J[\Edges{A}{C}]$. 

Letting $\bom{1}=A, \bom{2}=B$, and $\bom{3}=C$, this means that $\bom{1},\bom{2}$, and $\bom{3}$ are subpermutations of $\BO{1},\BO{2}$, and $\BO{3}$, respectively, with $|\bom{1}|=|\bom{2}|=|\bom{3}|=q$ such that no edge in $J[\edges{\bom{1}}{\bom{2}}]$ crosses an edge in $J[\edges{\bom{1}}{\bom{3}}]$, as required.
\end{proof}

With the proof of the Claim the proof of Theorem~\ref{thm:main} is now complete.
\end{proof}

\section{Proof of Proposition~\ref{pro:core}}\label{sec:core}

In a nutshell, Proposition~\ref{pro:core} claims that if we have permutations $\oA,\oB_1,\ldots,\oB_r$ of vertices of a graph, and a drawing whose restriction to $\Edges{\oA}{\oB_i}$ is a $\cycper{\oA\,\cdot\oB_i}$-drawing for each $i$, then there are subpermutations $A\preceq \oA,B_1\preceq\oB_1,\ldots,B_r\preceq\oB_r$ and a permutation $\pi$ or $[r]$ such that the restriction of the drawing to $\Edges{A}{B_1\cup\cdots\cup B_r}$ is a $\cycper{A\cdot B_{\pi(1)}\cdot \cdots \cdot B_{\pi(r)}}$-drawing. 

As a first step towards the proof of the proposition, let us start with two lemmas (namely Lemmas~\ref{lem:kon1} and~\ref{lem:kon}) that give conditions that guarantee outcomes of this character. 

\subsection{\outer drawings that combine into a single \outer drawing}

We start with the following lemma, which guarantees an outcome that closely resembles the conclusion of Proposition~\ref{pro:core}.

\begin{lemma}\label{lem:kon1}
{Let $A,C_1,\ldots,C_r$ be pairwise disjoint permutations of vertices in a graph $G$, such that every vertex in $A$ is adjacent to every vertex in $C_1\cup\cdots\cup C_r$. Let $I$ be a drawing of $G$. If $I[\edges{A}{(C_i \cup C_j)}]$ is a $\cycper{A\cdot C_i\cdot C_j}$-{\outerd} for all $1\le i < j \le r$, then $I[\edges{A}{(C_1\cup \,\,\cdots\,\,\cup C_r)}]$ is a $\cycper{A\cdot C_1\cdot \,\,\cdots\,\,\cdot C_r}$-\outerd.}
\end{lemma}

As in Negami's proof of Theorem~\ref{thm:negami2} above (see~\cite{negami}), in order to prove Lemma~\ref{lem:kon1} we make essential use of a result that gives conditions that guarantee that a drawing is \outer. This is Lemma~\ref{lem:outerplanar} below. This statement, established in~\cite{negami}, involves the notion of what we call a ``natural pair'' of permutations of vertices in a drawing. 


The concept of a natural pair relies on the notion of the rotation at a crossing. Since all drawings under consideration are simple, no three edges cross at a common point. Thus, every crossing $x$ involves exactly two edges $e=u_1u_2$ and $f=v_1v_2$. The {\em rotation} at $x$  is the cyclic permutation that records the clockwise cyclic order in which the edges towards $u_1,u_2,v_1$, and $v_2$ leave $x$. 
For an example we refer the reader to Figure~\ref{fig:255} and its caption. 

\begin{figure}[ht!]
\def\ta#1{{\Scale[1.8]{#1}}}
\def\tb#1{{\Scale[1.4]{#1}}}
\centering
\scalebox{0.4}{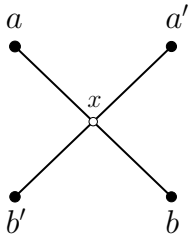}
\caption{The rotation at the crossing $x$ is $\cycper{a,a',b,b'}$.}
\label{fig:255}
\end{figure}

Now let $A,B$ be disjoint permutations of vertices in a graph $G$, such that every vertex of $A$ is adjacent to every vertex in $B$. Let $I$ be a drawing of $G$. We say that the pair $[A,B]$ is {\em natural in $I$} if the following holds: if $a,a'$ (respectively, $b,b'$) are vertices in $A$ (respectively, $B$) such that $a <_A a'$ (respectively, $b <_B b'$), then $ab$ and $a' b'$ cross each other in $I$, and the rotation at this crossing is $\cycper{a,a',b,b'}$ (as in Figure~\ref{fig:255}).

We note that this definition is symmetric in $A$ and $B$: the pair $[A,B]$ is natural in $I$ if and only if the pair $[B,A]$ is natural in $I$.

For an illustration of this concept we refer the reader back to Figure~\ref{fig:445}. As in that figure, let $\bom{1}=\perm{1(a_1),1(a_2),1(a_3),1(a_4)}$ and $\bom{2}=\perm{2(b_1),2(b_2),2(b_3),2(b_4))}$. Then the pair $[\bom{1},\bom{2}]$ is natural in the drawing in Figure~\ref{fig:445}(a); the pair $[\Rev{\bom{1}},\bom{2}]$ is natural in the drawing in Figure~\ref{fig:445}(b); the pair $[\bom{1},\Rev{\bom{2}}]$ is natural in the drawing in Figure~\ref{fig:445}(c); and the pair $[\Rev{\bom{1}},\Rev{\bom{2}}]$ is natural in the drawing in Figure~\ref{fig:445}(d).

The four drawings in Figure~\ref{fig:445} hint at a close relationship between naturality and \outerity. This is captured in the following statement, established by Negami in~\cite{negami}:

\begin{lemma}[{\cite[Lemma 5]{negami}}]\label{lem:outerplanar}
Let $A,B$ be disjoint permutations of vertices in a graph $G$, such that every vertex in $A$ is adjacent to every vertex in $B$. Let $I$ be a drawing of $G$. Then the pair $[A,B]$ is natural in $I$ if and only if $I[\edges{A}{B}]$ is a $\cycper{A\cdot B}$-\outerd.
\end{lemma}

With this tool in hand we are ready to prove Lemma~\ref{lem:kon1}.

\begin{proof}[Proof of Lemma~\ref{lem:kon1}]
In view of Lemma~\ref{lem:outerplanar}, the hypotheses of Lemma~\ref{lem:kon1} imply that ($*$) {\sl the pair $[A,C_i \cdot C_j]$ is natural in $I$, for all $1\le i < j \le r$}. Also in view of Lemma~\ref{lem:outerplanar} it suffices to show that $[A,C_1\cdot \,\,\cdots\,\,\cdot C_r]$ is natural in $I$.

We start by noting that since the pair $[A,C_i \cdot C_j]$ is natural in $I$ for all $1\le i < j \le r$, in particular it follows that ($**$) {\sl the pair $[A,C_i]$ is natural in $I$ for each $i\in[r]$}.

Let $a,a'$ be any vertices in $A$ such that $a<_A a'$, and let $u,v$ be any vertices in $C_1\cup \,\,\cdots \,\,\cup C_r$ such that $u<_{C_1\cdot \,\,\cdots \,\,\cdot C_r} v$. Our goal is to show that then ($\dag$) {\sl the edge $a u$ crosses the edge $a' v$ in $I$, and the rotation at this crossing in $I$ is $\cycper{a,a',u,v}$}.

Since $u<_{C_1\cdot \,\,\cdots \,\,\cdot C_r}v$ in $C_1\cdot \,\,\cdots \,\,\cdot C_r$, there are two cases to analyze: either (i) $u$ and $v$ are in the same $C_i$ for some $i\in[r]$; or (ii) $u\in C_i$ and $v\in C_j$, for some $1\le i < j \le r$. We conclude the proof by noting that if (i) holds then ($\dag$) follows from ($**$), and if (ii) holds then ($\dag$) follows from ($*$). 
\end{proof}


We close this subsection with a result whose proof also makes essential use of natural pairs. This crucial lemma gives conditions under which the union of two \outer drawings is a \outer drawing.

\begin{lemma}\label{lem:kon}
Let $A,B,C$ be pairwise disjoint permutations of vertices in a graph $G$, such that every vertex in $A$ is adjacent to every vertex in $B\cup C$. Let $I$ be a drawing of $G$. Suppose that $I[\edges{A}{B}]$ is a $\cycper{A\cdot B}$-{\outerd} and $I[\edges{A}{C}]$ is a $\cycper{A\cdot C}$-\outerd. Further suppose that if $a,a'$ are any vertices in $A$ such that $a <_A a'$, $b$ is any vertex in $B$, and $c$ is any vertex in $C$, then the edge $a b$ crosses the edge $a' c$ in $I$, and the rotation at this crossing is $\cycper{a,a',b,c}$. Then $I[\edges{A}{B\cup C}]$ is a $\cycper{A \cdot B \cdot C}$-\outerd.
\end{lemma}

\begin{proof}
In view of Lemma~\ref{lem:outerplanar}, the hypotheses imply that both $[A,B]$ and $[A,C]$ are natural in $I$. Also in view of Lemma~\ref{lem:outerplanar} it suffices to show that the pair $[A,B\cdot C]$ is natural in $I$.

Let $a,a'$ be any vertices in $A$ such that $a<_A a'$, and let $u,v$ be any vertices in $B\cdot C$ such that $u<_{B\cdot C}v$. Our goal is to show that then ($\dag$) {\sl the edge $a u$ crosses the edge $a' v$ in $I$, and the rotation at this crossing in $I$ is $\cycper{a,a',u,v}$}.

Since $u <_{B\cdot C}v$, there are three cases to analyze: either (i) both $u$ and $v$ are in $B$; (ii) both $u$ and $v$ are in $C$; or (iii) $u$ is in $B$ and $v$ is in $C$.

Suppose first that (i) holds. In this case the assumption that the pair $[A,B]$ is natural in $I$ implies ($\dag$). Similarly, if (ii) holds, then the assumption that the pair $[A,C]$ is natural in $I$ implies ($\dag$). Finally, if (iii) holds then by performing the relabellings $u\mapsto b$ and $v\mapsto c$ we see that ($\dag$) is an explicit assumption in the statement of the lemma.
\end{proof}


\subsection{Proof of Proposition~\ref{pro:core}}

As we are about to see, Proposition~\ref{pro:core} is an easy consequence of the following.

\begin{lemma}\label{lem:big}
Let $Q > q \ge 1$ be integers. Let $\oA,\oB_1,\ldots,\oB_r$ be pairwise disjoint permutations of vertices in a graph $G$, with $|\oA|=|\oB_1|=\cdots=|\oB_r|=Q$, such that every vertex in $\oA$ is adjacent to every vertex in $\oB_1\cup\cdots\cup \oB_r$. Suppose that $J$ is a drawing of $G$ such that $J[\edges{\oA}{\oB_i}]$ is a $\cycper{\oA \cdot \oB_i}$-{\outerd} for all $1\le i \le r$. If $Q$ is sufficiently large compared to $q$, then there exist subpermutations $A\preceq \oA$ and $B_i\preceq \oB_i$ for $i=1,\ldots,r$, with $|A|=|B_1|=\cdots=|B_{r}|=q$ such that the following hold for any distinct integers $i,j,k\in[r]$. 
\begin{enumerate}
\item[(I)] $J[\Edges{A}{B_i\cup B_j}]$ is either a  $\cycper{A\cdot B_i\cdot B_j}$-drawing or a $\cycper{A\cdot B_j\cdot B_i}$-\outerd.
\item[(II)] (Transitivity). If $J[\Edges{A}{B_i\cup B_j}]$ is a $\cycper{A\cdot B_i\cdot B_j}$-drawing and $J[\edges{A}{B_j\cup B_k}]$ is a $\cycper{A\cdot B_j\cdot B_k}$-drawing, then $J[\edges{A}{B_i\cup B_k}]$ is a $\cycper{A\cdot B_i\cdot B_k}$-drawing.
\end{enumerate}
\end{lemma}

Deferring the proof of this lemma for a moment, let us show that indeed it easily implies Proposition~\ref{pro:core}. 

\begin{proof}[Proof of Proposition~\ref{pro:core}, assuming Lemma~\ref{lem:big}]
The hypotheses of Proposition~\ref{pro:core} and Lemma~\ref{lem:big} are identical. Working under these hypotheses, Lemma~\ref{lem:big} guarantees that there exist subpermutations $A\preceq \oA$ and $B_i\preceq \oB_i$ for $i=1,\ldots,r$ with $|A|=|B_1|=\cdots=|B_{r}|=q$ such that (I) and (II) hold for any distinct integers $i,j,k\in[r]$. Put together, (I) and (II) imply the existence of a permutation $\perm{\pi(1),\ldots,\pi(r)}$ of $[r]$ such that $J[\edges{A}{\bigl(B_{\pi(i)} \cup B_{\pi(j)}\bigr)}]$ is a $\cycper{A \cdot B_{\pi(i)}\cdot B_{\pi(j)}}$-{\outerd} for all $1\le i < j \le r$. Thus the hypotheses of Lemma~\ref{lem:kon1} hold if we let $C_i=B_{\pi(i)}$ for $i=1,\ldots,r$, and so applying that lemma we obtain that $J[\edges{A}{\bigl( B_{1}\cup \,\,\cdots\,\,\cup B_{r}\bigr)}]=J[\edges{A}{\bigl( B_{\pi(1)}\cup \,\,\cdots\,\,\cup B_{\pi(r)}\bigr)}]$ is a $\cycper{A\cdot B_{\pi(1)}\cdot \,\,\cdots\,\,\cdot B_{\pi(r)}}$-\outerd.
\end{proof}

We devote the rest of the section to the proof of Lemma~\ref{lem:big}. 

\subsection{Proof of Lemma~\ref{lem:big}}

In the proof we use the notion of rotation at a vertex. Let $v$ be a vertex and $U$ a set of vertices in a graph, such that every vertex in $U$ is adjacent to $v$. If $D$ is a drawing of $G$, the {\em rotation} $\rotvU{v}{U}{D}$ {\em of $U$ at $v$ in $D$} is the cyclic permutation that records the clockwise cyclic order in which the edges incident with the vertices of $U$ leave $v$. For instance, in the drawing $D$ in Figure~\ref{fig:445}(a) (or (b)) for each $i=1,2,3,4$ we have that $\rotvU{1(a_i)}{\{2(b_1),2(b_2),2(b_4)\}}{D}=\cycper{2(b_1),2(b_2),2(b_4)}$. On the other hand, if $D$ is the drawing in Figure~\ref{fig:445}(c) (or (d)) then $\rotvU{1(a_i)}{\{2(b_1),2(b_2),2(b_4)\}}{D}=\cycper{2(b_4),2(b_2),2(b_1)}$, for $i=1,2,3,4$. 

\begin{proof}[Proof of Lemma~\ref{lem:big}]
We start by noting that in order to prove (I) it suffices to prove it for an arbitrary pair of distinct integers $i,j$, as then an iterative application of this yields the result for all pairs of distinct integers. Moreover, for notation simplicity instead of working with a pair $\oB_i,\oB_j$ of permutations we simply relabel $\oB_i$ with $\oB$ and $\oB_j$ with $\oC$.

Thus we assume that $\oA,\oB$, and $\oC$ are pairwise disjoint permutations of vertices in a graph $G$ with $|\oA|=|\oB|=|\oC|=Q$, such that every vertex in $\oA$ is adjacent to every vertex in $\oB\cup \oC$. We suppose that $J$ is a drawing of $G$ such that $J[\edges{\oA}{\oB}]$ is a $\cycper{\oA \cdot \oB}$-{\outerd} and $J[\edges{\oA}{\oC}]$ is a $\cycper{\oA \cdot \oC}$-{\outerd}. In order to prove (I) we must show the following.

\vglue 0.4 cm
\noindent{($\dag$) }{\sl If $Q$ is sufficiently large compared to $q$, then there exist subpermutations $A\preceq \oA,B\preceq \oB$, and $C\preceq\oC$ with $|A|=|B|=|C|=q$ such that $J[\Edges{A}{B\cup C}]$ is either a  $\cycper{A\cdot B\cdot C}$-drawing or a $\cycper{A\cdot C\cdot B}$-\outerd.}
\vglue 0.4 cm 

To prove ($\dag$) we start by constructing an auxiliary $4$-uniform hypergraph $\hh$, where each $4$-edge consists of two vertices in $\oA$, one vertex in $\oB$, and one vertex in $\oC$. We colour each $4$-edge in $\hh$ with one of five colours $\eta_0,\eta_1,\eta_2,\eta_3,\eta_4$, according to the following rules, illustrated in Figure~\ref{fig:355}. Let $\{a,a',b,c\}$ be a $4$-edge in $\hh$, where $a$ and $a'$ are in $\oA$ and $a<_{\oA}a'$, $b\in\oB$ and $c\in\oC$. Then:

\begin{itemize}

\item if $J[\edges{\{a,a'\}}{\{b,c\}}]$ has no crossings in $J$, then we colour $\{a,a',b,c\}$  with $\eta_0$;

\item for $i=1,2,3,4$, if $J[\edges{\{a,a'\}}{\{b,c\}}]$ has a crossing $x$ in $J$, and the rotation at $x$ is as in $\eta_i$ in Figure~\ref{fig:355}, then we colour $\{a,a',b,c\}$  with $\eta_i$.

\end{itemize}

\begin{figure}[ht!]
\def\ta#1{{\Scale[1.6]{#1}}}
\def\tta#1{{\Scale[1.8]{#1}}}
\def\tb#1{{\Scale[1.4]{#1}}}
\centering
\scalebox{0.4}{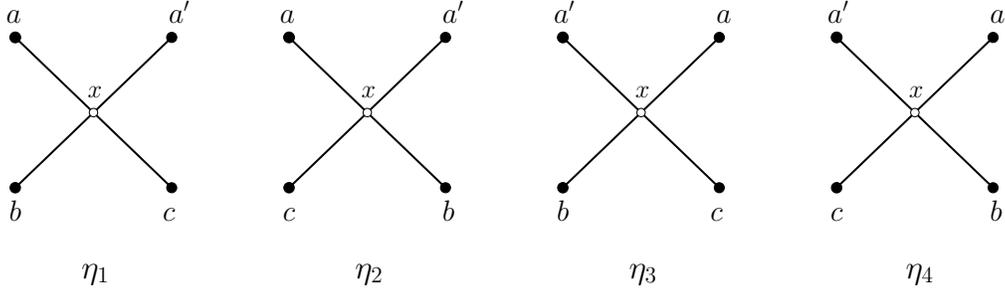}
\caption{Illustration of the rules for colouring each $4$-edge in $\hh$.}
\label{fig:355}
\end{figure}

Clearly, $\{a,a',b,c\}$ receives (exactly) one colour in $\{\eta_0,\eta_1,\eta_2,\eta_3,\eta_4\}$, as $J[\edges{\{a,a'\}}{\{b,c\}}]$ is a drawing of a $K_{2,2}$ and every simple drawing of $K_{2,2}$ has exactly zero or one crossings. If this drawing of $K_{2,2}$ has no crossings, it receives colour $\eta_0$. If it has one crossing, then the colour depends on the rotation at the crossing, according to the given rules.

By Ramsey's theorem (we can use for instance~\cite[Theorem 5]{ramseybook}) if $Q$ is sufficiently large compared to $q$ then there exist $A\preceq \oA, B\preceq\oB$, and $C\preceq \oC$ with $|A|=|B|=|C|=q$ such that all $4$-edges contained in $A\cup B\cup C$ have the same colour $\chi$. 

Note that the assumptions that $J[\edges{\oA}{\oB}]$ is a $\cycper{\oA\cdot \oB}$-{\outerd} and $J[\edges{\oA}{\oC}]$ is a $\cycper{\oA\cdot \oC}$-{\outerd} imply that $J[\edges{A}{B}]$ is a $\cycper{A\cdot B}$-{\outerd} and $J[\edges{A}{C}]$ is a $\cycper{A\cdot C}$-{\outerd}. 

We claim that $\chi$ is either $\eta_1$ or $\eta_2$, and that in either case we are done. To prove the latter assertion, suppose that $\chi$ is $\eta_1$. Let $a$ and $a'$ be any vertices in $A$ such that $a<_{A} a'$, let $b$ be any vertex in $B$, and let $c$ be any vertex in $C$. By assumption all $4$-edges in $\hh$ have colour $\eta_1$, and so $\{a,a',b,c\}$ has colour $\eta_1$. Therefore the edge $a c$ crosses the edge $a'\, b$ in $J$, and the rotation at this crossing in $J$ is $\cycper{a,a',c,b}$. In view of Lemma~\ref{lem:kon} it follows that $J[\edges{A}{(B\cup C)}]$ is a $\cycper{A\cdot C\cdot B}$-{\outerd}, and so we are done.

Totally analogous arguments show that if $\chi$ is $\eta_2$, then $J[\edges{A}{(B\cup C)}]$ is a $\cycper{A\cdot B\cdot C}$-{\outerd}, and so also in this case we are done.

Thus, in order to finish the proof of ($\dag$) it remains to show that $\chi$ is necessarily either $\eta_1$ or $\eta_2$. To prove this first let $a,a'$ and $a''$ be vertices in $A$ such that $a <_{A}a' <_{A} a''$. Let $b$ be any vertex in $B$, and let $c$ be any vertex in $C$. 

Since $J[\edges{A}{B}]$ is a $\cycper{A\cdot B}$-{\outerd} it follows that $\rotvU{b}{\{a,a',a''\}}{J}$ is $\cycper{a,a',a''}$. Similarly, since $J[\edges{A}{C}]$ is a $\cycper{A\cdot C}$-{\outerd} it follows that $\rotvU{c}{\{a,a',a''\}}{J}$ is also $\cycper{a,a',a''}$. Thus $b$ and $c$ have the same rotation in $J[\edges{\{a,a',a''\}}{\{b,c\}}]$.

Suppose first that $\chi=\eta_0$. This implies that $J[\edges{\{a,a',a''\}}{\{b,c\}}]$ has no crossings. On the other hand, it is straightforward to check that no drawing of $K_{2,3}$ in which the vertices in the $2$-class have the same rotation can have zero crossings. This contradiction implies that $\chi$ cannot be $\eta_0$. 

Therefore $\chi$ is $\eta_i$ for some $i\in\{1,2,3,4\}$. We note that in particular this implies that each $K_{2,2}$ in $J[\edges{\{a,a',a''\}}\{{b,c}\}]$ has one crossing.

We now invoke the following.

\begin{observation}\label{obs:k23two}
{\em If $L$ is a drawing of $K_{2,3}$ such that (i) the rotations at the vertices in the $2$-class are the same, and (ii) each $K_{2,2}$ has one crossing, then $L$ is isomorphic to one of the drawings in Figure~\ref{fig:1745}.}
\end{observation} 

\def\ta#1{{\Scale[3.0]{{#1}}}}
\def\tarm#1{{\Scale[3.0]{\text{\rm #1}}}}
\begin{figure}[ht!]
\centering
\scalebox{0.29}{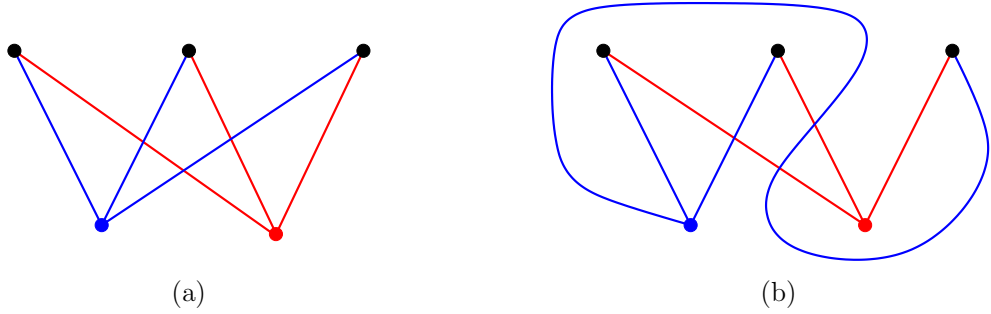}
\caption{The only two drawings (up to isomorphism) of $K_{2,3}$ in which the vertices in the $2$-class have the same rotation, and each $K_{2,2}$ has one crossing.}
\label{fig:1745}
\end{figure}

In view of this observation the previous remarks imply that $J[\edges{\{a,a',a''\}}{\{b,c\}}]$ is isomorphic to one of the drawings in Figure~\ref{fig:1745}. This leaves open only the question of which vertex (blue or red) is $b$ and which is $c$, and which black vertex is $a$, which one is $a'$, and which one is $a''$. We recall that by assumption all $4$-edges in $J[\edges{\{a,a',a''\}}{\{b,c\}}]$ have the same colour $\chi$. 

This yields $2!\cdot 3!=12$ potential labellings for each of the drawings in Figure~\ref{fig:1745}. It is a tedious but straightforward task to verify that the only two labellings that satisfy that all $4$-edges are of the same colour are the ones shown in Figure~\ref{fig:1750}. In the labelling in Figure~\ref{fig:1750}(a) all $4$-edges are of colour $\eta_1$, and in the labelling in (b) all edges are of colour $\eta_2$. Therefore $\chi$ must be either $\eta_1$ or $\eta_2$, as claimed. This completes the proof of ($\dag$), and hence of (I).

\def\ta#1{{\Scale[3.0]{{#1}}}}
\def\tarm#1{{\Scale[3.0]{\text{\rm #1}}}}
\begin{figure}[ht!]
\centering
\scalebox{0.29}{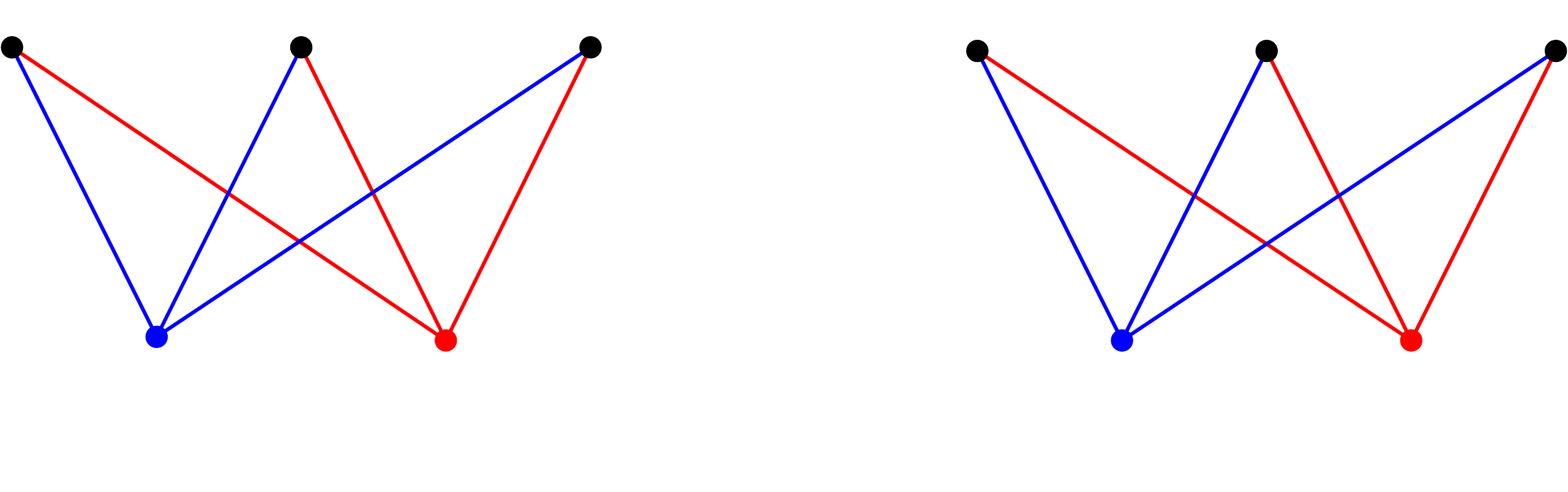}
\caption{Illustration of the conclusion of the proof of ($\dag$) in the proof of Lemma~\ref{lem:big}.}
\label{fig:1750}
\end{figure}

We start the proof of (II) by reviewing where we stand. Having proved (I) we have subpermutations $A,B_1,\ldots,B_r$ such that $J[\Edges{A}{B_i\cup B_j}]$ is either a $\cycper{A\cdot B_i\cdot B_j}$-drawing or a $\cycper{A\cdot B_j\cdot B_i}$-drawing for all distinct $i,j\in[r]$. 

Now (II) claims a transitivity property: if $i,j,k$ are any distinct integers in $[r]$ such that $J[\Edges{A}{B_i\cup B_j}]$ is a $\cycper{A\cdot B_i\cdot B_j}$-drawing and $J[\edges{A}{B_j\cup B_k}]$ is a $\cycper{A\cdot B_j\cdot B_k}$-drawing, then $J[\edges{A}{B_i\cup B_k}]$ is a $\cycper{A\cdot B_i\cdot B_k}$-drawing. 

As in the proof of (I), we may as well simplify the notation and prove the statement after relabeling $B_i\mapsto B, B_j\mapsto C$, and $B_k\mapsto D$. Thus in order to prove (II) we need to prove the following: 

\vglue 0.4 cm

\noindent ($\ddag$) If $J[\Edges{A}{B\cup C}]$ is a $\cycper{A\cdot B\cdot C}$-drawing and $J[\edges{A}{C\cup D}]$ is a $\cycper{A\cdot C\cdot D}$-drawing, then $J[\edges{A}{B\cup D}]$ is a $\cycper{A\cdot B\cdot D}$-drawing.

\vglue 0.4 cm

We prove ($\ddag$) by way of contradiction. We assume that (i) $J[\Edges{A}{B\cup C}]$ is a $\cycper{A\cdot B\cdot C}$-drawing and (ii) $J[\edges{A}{C\cup D}]$ is a $\cycper{A\cdot C\cdot D}$-drawing, and yet $J[\edges{A}{B\cup D}]$ is not a $\cycper{A\cdot B\cdot D}$-drawing. 

To obtain the required contradiction we start by noting that (I) implies that $J[\edges{A}{B\cup D}]$ is either a $\cycper{A\cdot B\cdot D}$-drawing or a $\cycper{A\cdot D\cdot B}$-drawing (recall that we relabelled $B_i\mapsto B$ and $B_k\mapsto D$). Thus the assumption that $J[\edges{A}{B\cup D}]$ is not a $\cycper{A\cdot B\cdot D}$-drawing implies that necessarily (iii) $J[\edges{A}{B\cup D}]$ is a $\cycper{A\cdot D\cdot B}$-drawing. Thus to prove ($\ddag$) we need to derive a contradiction from (i), (ii) and (iii).

Let $a,a',a''$ be vertices in $A$ such that $a<_{A}a'<_{A} a''$. Let $b$ be a vertex in $B$, let $c$ be a vertex in $C$, and let $d$ be a vertex in $D$.

Since by assumption $J[\edges{A}{(B \cup C)}]$ is a $\cycper{A\cdot B \cdot C}$-{\outerd} it follows that $J[\edges{\{a,a',a''\}}{\{b,c\}}]$ is isomorphic to the drawing in Figure~\ref{fig:1765}(a). 

\def\ta#1{{\Scale[3.0]{{#1}}}}
\def\tarm#1{{\Scale[3.0]{\text{\rm #1}}}}
\begin{figure}[ht!]
\centering
\scalebox{0.29}{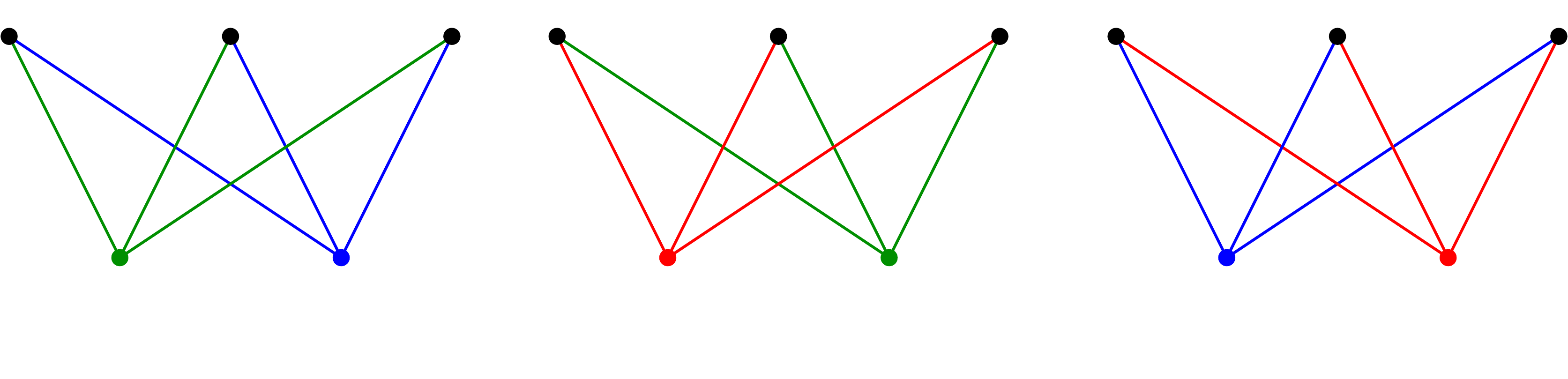}
\caption{Illustration of the conclusion of the proof of ($\ddag$) in the proof of Lemma~\ref{lem:big}.}
\label{fig:1765}
\end{figure}

Similarly, since $J[\edges{A}{(C \cup D)}]$ is a $\cycper{A\cdot C \cdot D}$-{\outerd} it follows that $J[\edges{\{a,a',a''\}}{\{c,d\}}]$ is isomorphic to the drawing in Figure~\ref{fig:1765}(b). 

Finally, since $J[\edges{A}{(B \cup D)}]$ is a $\cycper{A\cdot D \cdot B}$-{\outerd} it follows that $J[\edges{\{a,a',a''\}}{\{b,d\}}]$ is isomorphic to the drawing in Figure~\ref{fig:1765}(c). 

Thus $J[\edges{\{a,a',a''\}}{\{b,c,d\}}]$ must be a drawing of $K_{3,3}$ that contains drawings isomorphic to the drawings in Figure~\ref{fig:1765} as subdrawings, with the labels shown. But it is straightforward to verify that no such drawing of $K_{3,3}$ can possibly exist. This contradiction proves ($\ddag$) and hence (II).
\end{proof}

\section{Proof of Proposition~\ref{pro:core2}}\label{sec:core2}

Let us recall Proposition~\ref{pro:core2} for easy reference. 

\vglue 0.4 cm
\noindent \textbf{Proposition~\ref{pro:core2}. }
{\em Let $Q > q \ge 1$ be integers. Let $\oA, \oB$, and $\oC$ be pairwise disjoint permutations of vertices in a graph $G$, with $|\oA|=|\oB|=|\oC|=Q$, such that every vertex in $\oA$ is adjacent to every vertex in $\oB\cup\oC$. Let $J$ be a drawing of $G$ such that $J[\Edges{\oA}{\oB}]$ is a $\cycper{\oA\cdot\oB}$-drawing and $J[\Edges{\oA}{\oC}]$ is a $\cycper{\Rev{\oA}\cdot\oC}$-drawing. If $Q$ is sufficiently large compared to $q$, then there exist subpermutations $A\preceq \oA, B\preceq \oB$, and $C\preceq\oC$ with $|A|=|B|=|C|=q$ such that no edge in $J[\Edges{A}{B}]$ crosses an edge in $J[\Edges{A}{C}]$.}

\begin{proof}[Proof of Proposition~\ref{pro:core2}]
As in the proof of Lemma~\ref{lem:big}, we construct an auxiliary $4$-uniform hypergraph $\hh$, where each $4$-edge consists of two vertices in $a,a'$ in $\oA$, one vertex $b$ in $\oB$, and one vertex $c$ in $\oC$. We colour each $4$-edge $\{a,a',b,c\}$ in $\hh$ with one of the five colours $\eta_0,\eta_1,\eta_2,\eta_3$, or $\eta_4$, using the same rules as in the proof of Lemma~\ref{lem:big}.

By Ramsey's theorem if $Q$ is sufficiently large compared to $q$ then there exist $A\preceq \oA, B\preceq\oB$, and $C\preceq \oC$ with $|A|=|B|=|C|=q$ such that all $4$-edges contained in $A\cup B\cup C$ have the same colour $\chi$. 

We claim that $\chi$ is necessarily $\eta_0$. We note that this completes the proof, as this implies precisely that no edge in $J[\edges{A}{B}]$ crosses an edge in $J[\edges{A}{C}]$. 

For a contradiction, suppose that $\chi$ is $\eta_i$ for some $i\in\{1,2,3,4\}$.

Let $a,a'$ and $a''$ be vertices in $A$ such that $a<_{A }a' <_{A } a''$. Let $b$ be any vertex in $B$, and let $c$ be any vertex in $C$. 

Since $J[\edges{A}{B}]$ is a $\cycper{A\cdot B}$-drawing it follows that $\rotvU{b}{\{a,a',a''\}}{J}$ is $\cycper{a,a',a''}$, and since $J[\edges{A}{C}]$ is a $\cycper{\Rev{A}\cdot C}$-drawing it follows that $\rotvU{b}{\{a,a',a''\}}{J}$ is $\cycper{a'',a',a}$. 

We also note that the assumption that $\chi$ is $\eta_i$ for some $i\in\{1,2,3,4\}$ implies that each $K_{2,2}$ in $J[\edges{\{a,a',a''\}}{\{b,c\}}]$ has one crossing.

We now invoke the following observation, which follows from a tedious but straightforward exercise.

\begin{observation}\label{obs:k23one}
{\em If $L$ is a drawing of $K_{2,3}$ such that (i) the rotations at the vertices in the $2$-class are distinct, and (ii) each $K_{2,2}$ has one crossing, then $L$ is isomorphic to the drawing in Figure~\ref{fig:1775}.}
\end{observation} 

\def\ta#1{{\Scale[3.0]{{#1}}}}
\def\tarm#1{{\Scale[3.0]{\text{\rm #1}}}}
\begin{figure}[ht!]
\centering
\scalebox{0.29}{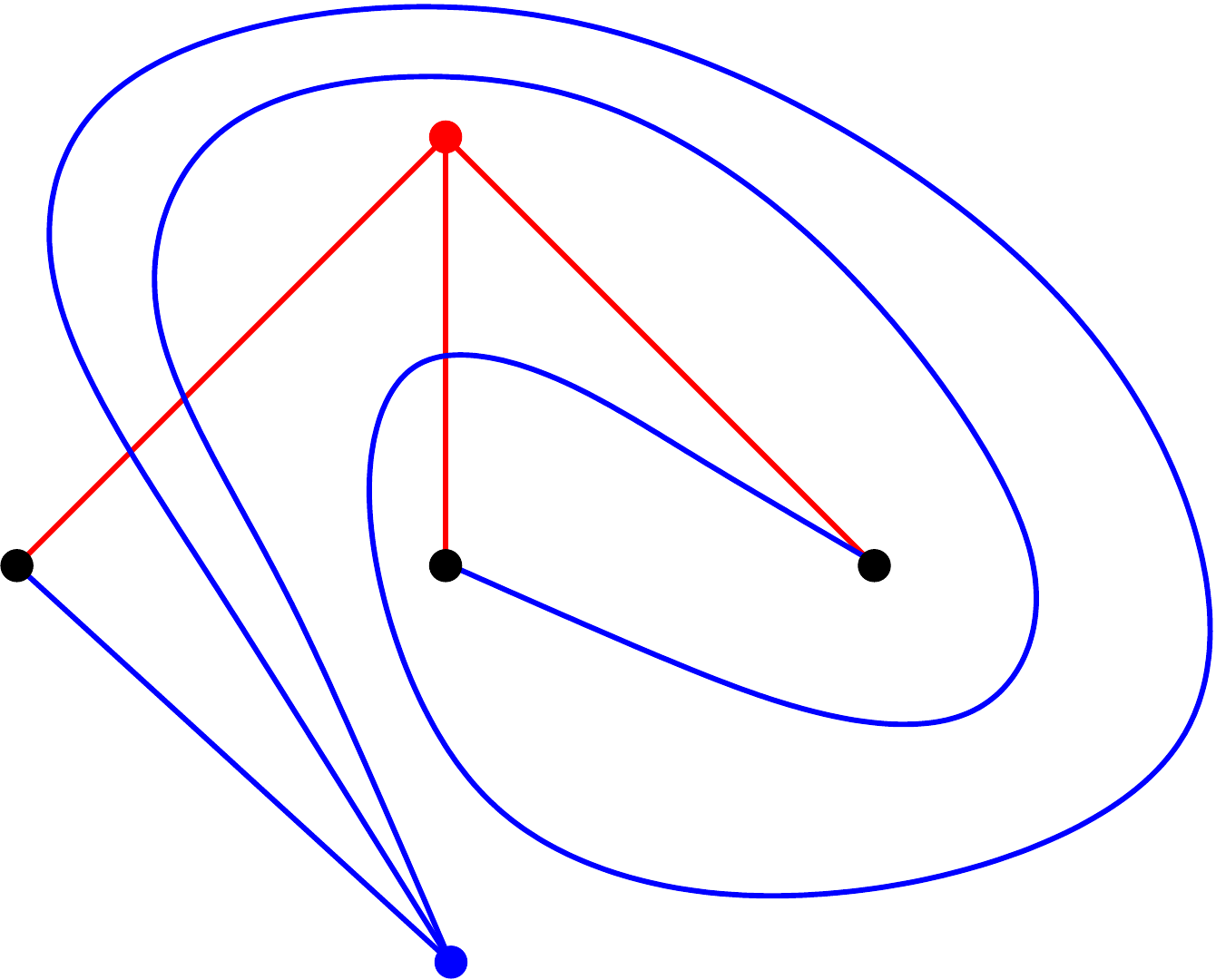}
\caption{The unique (up to isomorphism) drawing of $K_{2,3}$ in which the vertices in the $2$-class have different rotations, and each $K_{2,2}$ has one crossing.}
\label{fig:1775}
\end{figure}

The discussion before Observation~\ref{obs:k23one} implies that  $J[\edges{\{a,a',a''\}}{\{b,c}\}]$ satisfies the hypotheses of Observation~\ref{obs:k23one}. Therefore $J[\edges{\{a,a',a''\}}{\{b,c}\}]$ is isomorphic to the drawing in Figure~\ref{fig:1775}. This leaves open only the question of which vertex (blue or red) is $b$ and which is $c$, and which black vertex is $a$, which one is $a'$, and which one is $a''$. We recall that by assumption all $4$-edges in $J[\edges{\{a,a',a''\}}{\{b,c}\}]$ have the same colour $\chi$.

This yields $2!\cdot 3!=12$ potential labellings for the drawing in Figure~\ref{fig:1775}. It is a tedious but straightforward task to verify that no labelling satisfies that all $4$-edges are of the same colour. We have thus obtained the required contradiction.
\end{proof}

\section{Proof of Lemma~\ref{lem:templates1}}\label{sec:prooftemplates1}

The proof of the lemma relies on the following trivial observation. We recall that if $J$ is a \outer drawing, then saying that $J$ has bounding order $\rho$ is equivalent to saying that $J$ is a $\rho$-\outerd.

\begin{remark}\label{rem:outerp}
Suppose that $I$ and $J$ are \outer drawings of the same graph with the same bounding order. Then two edges cross each other in $I$ if and only if they cross each other in $J$.
\end{remark}

\begin{proof}[Proof of Lemma~\ref{lem:templates1}]

Let $C,D$ be canonical of drawings of $K_n^m$ with the same template $\Gamma=\bigl((1^+,1_{-}),$ $\ldots,$ $(m^+,m_{-})\bigr)$, and let $\sigma$ be the sign function of $\Gamma$. Let $i^+=\perm{i^1,\ldots,i^{|\sigma^+(i)|}}$ and $i_{-}=\perm{i_1,\ldots,i_{|\sigma_{-}(i)|}}$ for each $i\in[m]$. The assumption that $C$ and $D$ are canonical means that:

\begin{enumerate}

\item[(C1)] $C[\edges{\bom{i}}{(\bom{i^1}\cup\cdots\cup\bom{i^{|\sigma^+(i)|}})}]$ is a $\cycper{\bom{i} \cdot \sigma(i,i^1)\, \bom{i^1}\cdot\,\,\,\cdots\,\,\,\cdot \sigma(i,i^{|\sigma^+(i)|})\, \bom{i^{|\sigma^+(i)|}}}$-\outerd; 

\item[(D1)] $D[\edges{\bom{i}}{(\bom{i^1}\cup\cdots\cup\bom{i^{|\sigma^+(i)|}})}]$ is a $\cycper{\bom{i} \cdot \sigma(i,i^1)\, \bom{i^1}\cdot\,\,\,\cdots\,\,\,\cdot \sigma(i,i^{|\sigma^+(i)|})\, \bom{i^{|\sigma^+(i)|}}}$-\outerd; 

\item[(C2)] $C[\edges{\bom{i}}{({\bom{i_1}}\cup\cdots\cup\bom{i_{|\sigma_{-}(i)|}})}]$ is a $\cycper{\Rev{\bom{i}} \cdot \sigma(i,i_1)\, \bom{i_1}\cdot\,\,\,\cdots\,\,\,\cdot \sigma(i,i_{|\sigma_{-}(i)|})\, \bom{i_{|\sigma_{-}(i)|}}}$-\outerd; 

\item[(D2)] $D[\edges{\bom{i}}{({\bom{i_1}}\cup\cdots\cup\bom{i_{|\sigma_{-}(i)|}})}]$ is a $\cycper{\Rev{\bom{i}} \cdot \sigma(i,i_1)\, \bom{i_1}\cdot\,\,\,\cdots\,\,\,\cdot \sigma(i,i_{|\sigma_{-}(i)|})\, \bom{i_{|\sigma_{-}(i)|}}}$-\outerd; 

\item[(C3)] no edge in $C[\edges{\bom{i}}{(\bom{i^1}\cup\cdots\cup\bom{i^{|\sigma^+(i)|}})}]$ crosses an edge in $C[\edges{\bom{i}}{({\bom{i_1}}\cup\cdots\cup\bom{i_{|\sigma_{-}(i)|}})}]$; and

\item[(D3)] no edge in $D[\edges{\bom{i}}{(\bom{i^1}\cup\cdots\cup\bom{i^{|\sigma^+(i)|}})}]$ crosses an edge in $D[\edges{\bom{i}}{({\bom{i_1}}\cup\cdots\cup\bom{i_{|\sigma_{-}(i)|}})}]$.

\end{enumerate}

In order to prove that $C$ and $D$ are weakly isomorphic we let $e$ and $e'$ be two edges in $K_n^m$, and show that $e$ and $e'$ cross in $C$ if and only if they cross in $D$.

Let $i,j,k,\ell\in[m]$ be integers such that $e$ is in $\edges{\bom{i}}{\bom{j}}$ and $e'$ is in $\edges{\bom{k}}{\bom{\ell}}$. Thus $i\neq j$ and $k\neq \ell$, since no edge has both endvertices in the same partite class.

We break the analysis into two cases.

\vglue 0.2cm
\noindent{\sc Case 1. }{\sl $i,j,k,\ell$ are not all distinct from each other.}
\vglue 0.2cm

We may assume without loss of generality that $i=k$. There are two possibilities: either $j=\ell$ or $j\neq \ell$.

\vglue 0.2cm
\noindent{\sc Subcase 1.1. }{\sl $j=\ell$.}
\vglue 0.2cm

In this case $e$ and $e'$ are both in $\edges{\bom{i}}{\bom{j}}$. Now $j$ is either in $i^+$ or in $i_-$. We assume that $j\in i^+$, as the other possibility is handled similarly. 

Since $j\in i^+$, (C1) implies that $C[\Edges{\bom{i}}{\bom{j}}]$ is a $\cycper{\bom{i}\cdot\,\sigma(i,j)\,\bom{j}}$-\outerd. Similarly, (D1) implies that $D[\edges{\bom{i}}{\bom{j}}]$ is also a $\cycper{\bom{i}\cdot\,\sigma(i,j)\,\bom{j}}$-\outerd. Thus, Remark~\ref{rem:outerp} implies that $e$ and $e'$ cross in $C$ if and only if they cross in $D$.

\vglue 0.2cm
\noindent{\sc Subcase 1.2. }{\sl $j\neq\ell$.}
\vglue 0.2cm

In this case there are four possibilities for $j$ and $\ell$, as each of them may belong to either $i^+$ or to $i_-$.

Suppose that they both belong to $i^+$, and that $j<_{i^+}\ell$ (the case in which $\ell<_{i^+}j$ is handled in a totally analogous manner). Then (C1) implies that $C[\edges{\bom{i}}{(\bom{j}\cup\bom{\ell})}]$ is a $\cycper{\bom{i}\cdot\,\sigma(i,j)\,\bom{j}\cdot\,\sigma(i,\ell)\,\bom{\ell}}$-\outerd. Similarly, (D1) implies that $D[\edges{\bom{i}}{(\bom{j}\cup\bom{\ell})}]$ is also a $\cycper{\bom{i}\cdot\,\sigma(i,j)\,\bom{j}\cdot\,\sigma(i,\ell)\,\bom{\ell}}$-\outerd. Thus, Remark~\ref{rem:outerp} implies that $e$ and $e'$ cross in $C$ if and only if they cross in $D$.

The case in which $j$ and $\ell$ both belong to $i_-$ is similarly handled, using (C2) and (D2) instead of (C1) and (D1).

Suppose now that one of $j$ and $\ell$ belongs to $i^+$, and the other belongs to $i_-$. Without loss of generality, suppose that $j\in i^+$ and $\ell\in i_-$. In this case $e$ is in $C[\edges{\bom{i}}{(\bom{i^1}\cup\cdots\cup\bom{i^{|\sigma^+(i)|}})}]$ and $e'$ is in $C[\edges{\bom{i}}{({\bom{i_1}}\cup\cdots\cup\bom{i_{|\sigma_{-}(i)|}})}]$, and so by (C3) $e$ and $e'$ do not cross each other in $C$.

Similarly, the assumptions that $j\in i^+$ and $\ell\in i_-$ imply that $e$ is in $D[\edges{\bom{i}}{(\bom{i^1}\cup\cdots\cup\bom{i^{|\sigma^+(i)|}})}]$ and $e'$ is in $D[\edges{\bom{i}}{({\bom{i_1}}\cup\cdots\cup\bom{i_{|\sigma_{-}(i)|}})}]$, and so by (D3) $e$ and $e'$ do not cross each other in $D$. Thus in this case we are also done, since $e$ and $e'$ cross neither in $C$ nor in $D$.

\vglue 0.2cm
\noindent{\sc Case 2. }{\sl $i,j,k,\ell$ are all distinct from each other.}
\vglue 0.2cm

Label $v_i$ the endvertex of $e$ that is in $\bom{i}$ and label $v_j$ the endvertex of $e$ that is in $\bom{j}$. Similarly, label $v_k$ the endvertex of $e'$ that is in $\bom{k}$ and label $v_\ell$ the endvertex of $e'$ that is in $\bom{\ell}$. Finally, for each $p\in[m]\setminus\{i,j,k,\ell\}$ let $v_p$ be any vertex in partite class $\bom{p}$.

We consider $C[\{v_1,v_2,\ldots,v_m\}]$, the subdrawing of $C$ induced by the vertices $v_1,\ldots,v_m$ and the edges that have both endvertices in $\{v_1,\ldots,v_m\}$. This is a drawing of a complete graph on $m$ vertices labelled $v_1,v_2,\ldots,v_m$. It is straightforward to verify that (C1),(C2), and (C3) together imply that for each $i\in[m]$, $\rotvU{v_i}{\{v_1,\ldots,v_{i-1},v_{i+1},\ldots,v_m\}}{C}=\cycper{v_{i^1},\ldots,v_{i^{|\sigma^+(i)|}},v_{i_1},\ldots,v_{i_{|\sigma_{-}(i)|}}}$.

We now consider $D[\{v_1,v_2,\ldots,v_m\}]$, which is also a drawing of a complete graph on $m$ vertices labelled $v_1,v_2,\ldots,v_m$. It is straightforward to verify that (D1),(D2), and (D3) together imply that for each $i\in[m]$, $\rotvU{v_i}{\{v_1,\ldots,v_{i-1},v_{i+1},\ldots,v_m\}}{D}=\cycper{v_{i^1},\ldots,v_{i^{|\sigma^+(i)|}},v_{i_1},\ldots,v_{i_{|\sigma_{-}(i)|}}}$.

Thus, $C[\{v_1,v_2,\ldots,v_m\}]$ and $D[\{v_1,v_2,\ldots,v_m\}]$ are drawings of complete graphs on $m$ vertices with the same labels and with identical rotation systems. It follows from~\cite{kynclenumeration} that two edges cross in the former drawing if and only if they cross in the latter drawing. In particular, $e$ (whose endvertices are $v_i$ and $v_j$) and $e'$ (whose endvertices are $v_k$ and $v_\ell$) cross each other in $C$ if and only if they cross each other in $D$.
\end{proof}

\section{Realizable and non-realizable templates}\label{sec:realizable}

As we pointed out in Section~\ref{sub:realizable}, Theorem~\ref{thm:main} fully identifies the unavoidable drawings of $K_n^m$, in the sense that it establishes that they are all canonical. On the other hand, it does not address an essential question: which templates arise from a canonical drawing? That is, which templates are realizable? We now fully answer this question.


\begin{lemma}\label{lemma:realizable}
Let $m,n$ be positive integers. Let $\Gamma=\bigl((1^+,1_{-}),\ldots,(m^+,m_{-})\bigr)$ be a template, where $i^+=\perm{i^1,\ldots,i^{|\sigma^+(i)|}}$ and $i_-=\perm{i_1,\ldots,i_{|\sigma_{-}(i)|}}$ for each $i\in[m]$. Then there is a canonical drawing of $K_n^m$ with template $\Gamma$ if and only if $\{\cycper{1^+\cdot 1_-},\ldots,\cycper{m^+\cdot m_-}\}$ is the rotation system of a drawing of the complete graph $K_m$ with the vertices labelled $1,\ldots,m$.
\end{lemma}

\begin{proof}
To illustrate the ``if'' part we refer the reader to Figures~\ref{fig:180} and~\ref{fig:205}. These figures illustrate how to construct a canonical drawing of $K_3^5$ from the template $\Gamma_5=\bigl((1^+,1_-),\ldots,(5^+,5_-)\bigr)$, where $1^+=\perm{3,2},1_-=\perm{4,5}$, $2^+=\perm{}$, $2_-=\perm{3,1,5,4}$,  $3^+=\perm{1,2}$, $3_-=\perm{5,4}$, $4^+=\perm{2}$, $4_-=\perm{3,5,1}$, $5^+=\perm{3,1,4}$, and $5_-=\perm{2}$.

\def\te#1{{\Scale[4.2]{#1}}}
\begin{figure}[ht!]
\centering
\scalebox{0.2}{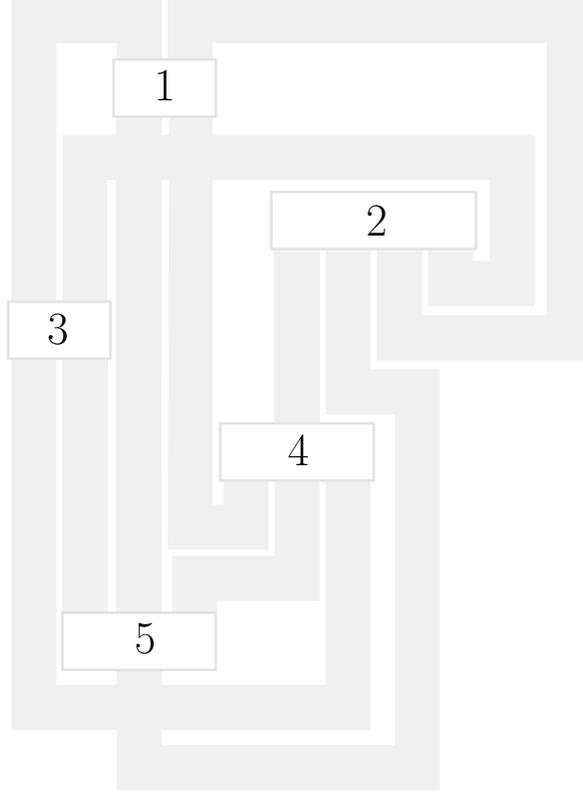}
\caption{A drawing of $K_5$ where the vertices are represented by rectangles and the edges are thick segments that join the rectangles. This drawing serves as a roadmap for the canonical drawing of $K_3^5$ in Figure~\ref{fig:180}.}
\label{fig:205}
\end{figure}

Note that $\{\cycper{1^+\cdot 1_-},\ldots,\cycper{5^+\cdot 5_-}\}$ is indeed the rotation system of a drawing of the complete graph $K_5$ with the vertices labelled $1,\ldots,5$: the drawing in Figure~\ref{fig:205} has precisely this rotation system. The reason to draw $K_5$ in such an unconventional way will be clear in a moment.

In general, suppose we are given a template $\Gamma=\bigl((1^+,1_{-}),\ldots,(m^+,$ $m_{-})\bigr)$, where $i^+=\perm{i^1,\ldots,i^{|\sigma^+(i)|}}$ and $i_-=\perm{i_1,\ldots,i_{|\sigma_{-}(i)|}}$ for each $i\in[m]$, and $\{\cycper{1^+\cdot 1_-},\ldots,\cycper{m^+\cdot m_-}\}$ is the rotation system of a drawing of the complete graph $K_m$ with the vertices labelled $1,\ldots,m$.

To obtain a canonical drawing of $K_n^m$ with template $\Gamma$ we first construct a drawing of $K_m$ (such as the one in Figure~\ref{fig:205}) in which the vertices are represented as axis-parallel rectangles, and the edges are represented by thick segments. The edges (segments) incident with the $i$-th rectangle $R_i$ leave $R_i$ according to the following rules. First, the segments that leave to rectangles $R_{i^1},\ldots,R_{i^{|\sigma^+(i)|}}$ intersect $R_i$ on its top side, and the intersections occur in this order from left to right. Finally, the segments that leave to rectangles $R_{i_1},\ldots,R_{i_{|\sigma_{-}(i)|}}$ intersect  $R_i$ on its bottom side, and the intersections occur in this order from right to left. 

We refer the reader again to Figure~\ref{fig:205}, where we illustrate a drawing of $K_5$ based on the template $\Gamma_5$ following these rules.

For each $i\in [m]$ we let $\bom{i}$ be the partite class (permutation) $\perm{i(1),\ldots,i(n)}$, and place these vertices horizontally inside $R_i$, in the order $i(n),\ldots,i(1)$ from left to right. See Figure~\ref{fig:180} for an illustration of this process for the drawing of $K_3^5$ obtained from the template $\Gamma_5$.

Finally, for each distinct $i,j\in[m]$ we route the edges in $\edges{\bom{i}}{\bom{j}}$ through the segment that joins rectangles $R_i$ and $R_j$, as shown in Figure~\ref{fig:180}.

At the end of this process we obtain a canonical drawing of $K_n^m$ with template $\Gamma$.

Regarding the ``only if'' part, suppose that $C$ is a canonical drawing of $K_n^m$ with template $\Gamma=\bigl((1^+,1_{-}),\ldots,(m^+,m_{-})\bigr)$. For each $i\in [m]$, choose any vertex in partite class $\bom{i}$ and label it simply with $i$. 

Thus, $C[\{1,\ldots,m\}]$ is a drawing of a complete graph $K_m$ with the vertices labelled $1,\ldots,m$. For each $i\in[m]$, the rotation at vertex $i$ is $\cycper{i^+\cdot i_-}$. Thus, $\{\cycper{1^+\cdot 1_-},\ldots,\cycper{m^+\cdot m_-}\}$ is the rotation system of a drawing of the complete graph $K_m$ with the vertices labelled $1,\ldots,m$.
\end{proof}


\section*{Acknowledgments}

The first author is supported in part by NSF grants DMS-1764123 and RTG DMS-1937241, FRG DMS-2152488, the Arnold O. Beckman Research Award (UIUC Campus Research Board RB 24012), and  by IBS-R029-C4.


\bibliographystyle{abbrv}
\bibliography{refs.bib}


\end{document}

%% file: harborthb.pdf_t
\begin{picture}(0,0)%
\includegraphics{harborthb.pdf}%
\end{picture}%
\setlength{\unitlength}{4144sp}%
\begingroup\makeatletter\ifx\SetFigFont\undefined%
\gdef\SetFigFont#1#2#3#4#5{%
  \reset@font\fontsize{#1}{#2pt}%
  \fontfamily{#3}\fontseries{#4}\fontshape{#5}%
  \selectfont}%
\fi\endgroup%
\begin{picture}(13065,3649)(-6689,-3185)
\put(-5084,-3121){\makebox(0,0)[lb]{\smash{{\SetFigFont{17}{20.4}{\familydefault}{\mddefault}{\updefault}{\color[rgb]{0,0,0}$\ta{C_6}$}%
}}}}
\put(3691,-3121){\makebox(0,0)[lb]{\smash{{\SetFigFont{17}{20.4}{\familydefault}{\mddefault}{\updefault}{\color[rgb]{0,0,0}$\ta{T_6}$}%
}}}}
\end{picture}%

%% file: 165f.pdf_t
\begin{picture}(0,0)%
\includegraphics{165f.pdf}%
\end{picture}%
\setlength{\unitlength}{4144sp}%
\begingroup\makeatletter\ifx\SetFigFont\undefined%
\gdef\SetFigFont#1#2#3#4#5{%
  \reset@font\fontsize{#1}{#2pt}%
  \fontfamily{#3}\fontseries{#4}\fontshape{#5}%
  \selectfont}%
\fi\endgroup%
\begin{picture}(4897,3946)(-554,-3635)
\put(2611,164){\makebox(0,0)[lb]{\smash{{\SetFigFont{12}{14.4}{\familydefault}{\mddefault}{\updefault}{\color[rgb]{0,0,0}$\ta{b_3}$}%
}}}}
\put(1621,164){\makebox(0,0)[lb]{\smash{{\SetFigFont{12}{14.4}{\familydefault}{\mddefault}{\updefault}{\color[rgb]{0,0,0}$\ta{b_2}$}%
}}}}
\put(721,164){\makebox(0,0)[lb]{\smash{{\SetFigFont{12}{14.4}{\familydefault}{\mddefault}{\updefault}{\color[rgb]{0,0,0}$\ta{b_1}$}%
}}}}
\put(3421,164){\makebox(0,0)[lb]{\smash{{\SetFigFont{12}{14.4}{\familydefault}{\mddefault}{\updefault}{\color[rgb]{0,0,0}$\ta{b_4}$}%
}}}}
\put(766,-3571){\makebox(0,0)[lb]{\smash{{\SetFigFont{12}{14.4}{\familydefault}{\mddefault}{\updefault}{\color[rgb]{0,0,0}$\ta{w_1}$}%
}}}}
\put(1666,-3571){\makebox(0,0)[lb]{\smash{{\SetFigFont{12}{14.4}{\familydefault}{\mddefault}{\updefault}{\color[rgb]{0,0,0}$\ta{w_2}$}%
}}}}
\put(2656,-3571){\makebox(0,0)[lb]{\smash{{\SetFigFont{12}{14.4}{\familydefault}{\mddefault}{\updefault}{\color[rgb]{0,0,0}$\ta{w_3}$}%
}}}}
\put(3466,-3571){\makebox(0,0)[lb]{\smash{{\SetFigFont{12}{14.4}{\familydefault}{\mddefault}{\updefault}{\color[rgb]{0,0,0}$\ta{w_4}$}%
}}}}
\put(-539,-1681){\makebox(0,0)[lb]{\smash{{\SetFigFont{12}{14.4}{\familydefault}{\mddefault}{\updefault}{\color[rgb]{0,0,0}$\ta{\omega}$}%
}}}}
\end{picture}%

%% file: 445a.pdf_t
\begin{picture}(0,0)%
\includegraphics{445a.pdf}%
\end{picture}%
\setlength{\unitlength}{4144sp}%
\begingroup\makeatletter\ifx\SetFigFont\undefined%
\gdef\SetFigFont#1#2#3#4#5{%
  \reset@font\fontsize{#1}{#2pt}%
  \fontfamily{#3}\fontseries{#4}\fontshape{#5}%
  \selectfont}%
\fi\endgroup%
\begin{picture}(11467,10897)(6691,-10520)
\put(6706,-1681){\makebox(0,0)[lb]{\smash{{\SetFigFont{12}{14.4}{\familydefault}{\mddefault}{\updefault}{\color[rgb]{0,0,0}$\ta{\omega}$}%
}}}}
\put(9361,-4381){\makebox(0,0)[lb]{\smash{{\SetFigFont{12}{14.4}{\familydefault}{\mddefault}{\updefault}{\color[rgb]{0,0,0}$\tarm{(a)}$}%
}}}}
\put(7921,-3616){\makebox(0,0)[lb]{\smash{{\SetFigFont{12}{14.4}{\familydefault}{\mddefault}{\updefault}{\color[rgb]{0,0,0}$\tabfour$}%
}}}}
\put(7831,254){\makebox(0,0)[lb]{\smash{{\SetFigFont{12}{14.4}{\familydefault}{\mddefault}{\updefault}{\color[rgb]{0,0,0}$\taaone$}%
}}}}
\put(8686,254){\makebox(0,0)[lb]{\smash{{\SetFigFont{12}{14.4}{\familydefault}{\mddefault}{\updefault}{\color[rgb]{0,0,0}$\taatwo$}%
}}}}
\put(9586,254){\makebox(0,0)[lb]{\smash{{\SetFigFont{12}{14.4}{\familydefault}{\mddefault}{\updefault}{\color[rgb]{0,0,0}$\taathree$}%
}}}}
\put(10486,254){\makebox(0,0)[lb]{\smash{{\SetFigFont{12}{14.4}{\familydefault}{\mddefault}{\updefault}{\color[rgb]{0,0,0}$\taafour$}%
}}}}
\put(10531,-3616){\makebox(0,0)[lb]{\smash{{\SetFigFont{12}{14.4}{\familydefault}{\mddefault}{\updefault}{\color[rgb]{0,0,0}$\tabone$}%
}}}}
\put(9631,-3616){\makebox(0,0)[lb]{\smash{{\SetFigFont{12}{14.4}{\familydefault}{\mddefault}{\updefault}{\color[rgb]{0,0,0}$\tabtwo$}%
}}}}
\put(8731,-3616){\makebox(0,0)[lb]{\smash{{\SetFigFont{12}{14.4}{\familydefault}{\mddefault}{\updefault}{\color[rgb]{0,0,0}$\tabthree$}%
}}}}
\put(13366,-1681){\makebox(0,0)[lb]{\smash{{\SetFigFont{12}{14.4}{\familydefault}{\mddefault}{\updefault}{\color[rgb]{0,0,0}$\ta{\omega}$}%
}}}}
\put(14581,-3616){\makebox(0,0)[lb]{\smash{{\SetFigFont{12}{14.4}{\familydefault}{\mddefault}{\updefault}{\color[rgb]{0,0,0}$\tabfour$}%
}}}}
\put(17191,-3616){\makebox(0,0)[lb]{\smash{{\SetFigFont{12}{14.4}{\familydefault}{\mddefault}{\updefault}{\color[rgb]{0,0,0}$\tabone$}%
}}}}
\put(16291,-3616){\makebox(0,0)[lb]{\smash{{\SetFigFont{12}{14.4}{\familydefault}{\mddefault}{\updefault}{\color[rgb]{0,0,0}$\tabtwo$}%
}}}}
\put(15391,-3616){\makebox(0,0)[lb]{\smash{{\SetFigFont{12}{14.4}{\familydefault}{\mddefault}{\updefault}{\color[rgb]{0,0,0}$\tabthree$}%
}}}}
\put(16021,-4381){\makebox(0,0)[lb]{\smash{{\SetFigFont{12}{14.4}{\familydefault}{\mddefault}{\updefault}{\color[rgb]{0,0,0}$\tarm{(b)}$}%
}}}}
\put(14491,254){\makebox(0,0)[lb]{\smash{{\SetFigFont{12}{14.4}{\familydefault}{\mddefault}{\updefault}{\color[rgb]{0,0,0}$\taafour$}%
}}}}
\put(15346,254){\makebox(0,0)[lb]{\smash{{\SetFigFont{12}{14.4}{\familydefault}{\mddefault}{\updefault}{\color[rgb]{0,0,0}$\taathree$}%
}}}}
\put(16246,254){\makebox(0,0)[lb]{\smash{{\SetFigFont{12}{14.4}{\familydefault}{\mddefault}{\updefault}{\color[rgb]{0,0,0}$\taatwo$}%
}}}}
\put(17146,254){\makebox(0,0)[lb]{\smash{{\SetFigFont{12}{14.4}{\familydefault}{\mddefault}{\updefault}{\color[rgb]{0,0,0}$\taaone$}%
}}}}
\put(6706,-7756){\makebox(0,0)[lb]{\smash{{\SetFigFont{12}{14.4}{\familydefault}{\mddefault}{\updefault}{\color[rgb]{0,0,0}$\ta{\omega}$}%
}}}}
\put(7831,-5821){\makebox(0,0)[lb]{\smash{{\SetFigFont{12}{14.4}{\familydefault}{\mddefault}{\updefault}{\color[rgb]{0,0,0}$\taaone$}%
}}}}
\put(8686,-5821){\makebox(0,0)[lb]{\smash{{\SetFigFont{12}{14.4}{\familydefault}{\mddefault}{\updefault}{\color[rgb]{0,0,0}$\taatwo$}%
}}}}
\put(9586,-5821){\makebox(0,0)[lb]{\smash{{\SetFigFont{12}{14.4}{\familydefault}{\mddefault}{\updefault}{\color[rgb]{0,0,0}$\taathree$}%
}}}}
\put(10486,-5821){\makebox(0,0)[lb]{\smash{{\SetFigFont{12}{14.4}{\familydefault}{\mddefault}{\updefault}{\color[rgb]{0,0,0}$\taafour$}%
}}}}
\put(9361,-10456){\makebox(0,0)[lb]{\smash{{\SetFigFont{12}{14.4}{\familydefault}{\mddefault}{\updefault}{\color[rgb]{0,0,0}$\tarm{(c)}$}%
}}}}
\put(13366,-7756){\makebox(0,0)[lb]{\smash{{\SetFigFont{12}{14.4}{\familydefault}{\mddefault}{\updefault}{\color[rgb]{0,0,0}$\ta{\omega}$}%
}}}}
\put(16021,-10456){\makebox(0,0)[lb]{\smash{{\SetFigFont{12}{14.4}{\familydefault}{\mddefault}{\updefault}{\color[rgb]{0,0,0}$\tarm{(d)}$}%
}}}}
\put(7921,-9691){\makebox(0,0)[lb]{\smash{{\SetFigFont{12}{14.4}{\familydefault}{\mddefault}{\updefault}{\color[rgb]{0,0,0}$\tabone$}%
}}}}
\put(8731,-9691){\makebox(0,0)[lb]{\smash{{\SetFigFont{12}{14.4}{\familydefault}{\mddefault}{\updefault}{\color[rgb]{0,0,0}$\tabtwo$}%
}}}}
\put(9631,-9691){\makebox(0,0)[lb]{\smash{{\SetFigFont{12}{14.4}{\familydefault}{\mddefault}{\updefault}{\color[rgb]{0,0,0}$\tabthree$}%
}}}}
\put(10531,-9691){\makebox(0,0)[lb]{\smash{{\SetFigFont{12}{14.4}{\familydefault}{\mddefault}{\updefault}{\color[rgb]{0,0,0}$\tabfour$}%
}}}}
\put(14581,-9691){\makebox(0,0)[lb]{\smash{{\SetFigFont{12}{14.4}{\familydefault}{\mddefault}{\updefault}{\color[rgb]{0,0,0}$\tabone$}%
}}}}
\put(15391,-9691){\makebox(0,0)[lb]{\smash{{\SetFigFont{12}{14.4}{\familydefault}{\mddefault}{\updefault}{\color[rgb]{0,0,0}$\tabtwo$}%
}}}}
\put(16291,-9691){\makebox(0,0)[lb]{\smash{{\SetFigFont{12}{14.4}{\familydefault}{\mddefault}{\updefault}{\color[rgb]{0,0,0}$\tabthree$}%
}}}}
\put(17191,-9691){\makebox(0,0)[lb]{\smash{{\SetFigFont{12}{14.4}{\familydefault}{\mddefault}{\updefault}{\color[rgb]{0,0,0}$\tabfour$}%
}}}}
\put(14491,-5821){\makebox(0,0)[lb]{\smash{{\SetFigFont{12}{14.4}{\familydefault}{\mddefault}{\updefault}{\color[rgb]{0,0,0}$\taafour$}%
}}}}
\put(15346,-5821){\makebox(0,0)[lb]{\smash{{\SetFigFont{12}{14.4}{\familydefault}{\mddefault}{\updefault}{\color[rgb]{0,0,0}$\taathree$}%
}}}}
\put(16246,-5821){\makebox(0,0)[lb]{\smash{{\SetFigFont{12}{14.4}{\familydefault}{\mddefault}{\updefault}{\color[rgb]{0,0,0}$\taatwo$}%
}}}}
\put(17146,-5821){\makebox(0,0)[lb]{\smash{{\SetFigFont{12}{14.4}{\familydefault}{\mddefault}{\updefault}{\color[rgb]{0,0,0}$\taaone$}%
}}}}
\end{picture}%

%% file: 175l.pdf_t
\begin{picture}(0,0)%
\includegraphics{175l.pdf}%
\end{picture}%
\setlength{\unitlength}{4144sp}%
\begingroup\makeatletter\ifx\SetFigFont\undefined%
\gdef\SetFigFont#1#2#3#4#5{%
  \reset@font\fontsize{#1}{#2pt}%
  \fontfamily{#3}\fontseries{#4}\fontshape{#5}%
  \selectfont}%
\fi\endgroup%
\begin{picture}(17526,23717)(-4384,-18332)
\put(-3374,-4651){\makebox(0,0)[lb]{\smash{{\SetFigFont{20}{24.0}{\familydefault}{\mddefault}{\updefault}{\color[rgb]{0,0,0}$\tq{3(2)}$}%
}}}}
\put(6436,-1006){\makebox(0,0)[lb]{\smash{{\SetFigFont{20}{24.0}{\familydefault}{\mddefault}{\updefault}{\color[rgb]{0,0,0}$\tq{2(2)}$}%
}}}}
\put(-944,-14146){\makebox(0,0)[lb]{\smash{{\SetFigFont{20}{24.0}{\familydefault}{\mddefault}{\updefault}{\color[rgb]{0,0,0}$\tq{5(2)}$}%
}}}}
\put(3781,-8161){\makebox(0,0)[lb]{\smash{{\SetFigFont{20}{24.0}{\familydefault}{\mddefault}{\updefault}{\color[rgb]{0,0,0}$\tq{4(2)}$}%
}}}}
\put(-1034,2594){\makebox(0,0)[lb]{\smash{{\SetFigFont{20}{24.0}{\familydefault}{\mddefault}{\updefault}{\color[rgb]{0,0,0}$\tq{1(3)}$}%
}}}}
\put(-224,2594){\makebox(0,0)[lb]{\smash{{\SetFigFont{20}{24.0}{\familydefault}{\mddefault}{\updefault}{\color[rgb]{0,0,0}$\tq{1(2)}$}%
}}}}
\put(1351,2594){\makebox(0,0)[lb]{\smash{{\SetFigFont{20}{24.0}{\familydefault}{\mddefault}{\updefault}{\color[rgb]{0,0,0}$\tq{1(1)}$}%
}}}}
\put(4861,-1006){\makebox(0,0)[lb]{\smash{{\SetFigFont{20}{24.0}{\familydefault}{\mddefault}{\updefault}{\color[rgb]{0,0,0}$\tq{2(3)}$}%
}}}}
\put(7921,-1006){\makebox(0,0)[lb]{\smash{{\SetFigFont{20}{24.0}{\familydefault}{\mddefault}{\updefault}{\color[rgb]{0,0,0}$\tq{2(1)}$}%
}}}}
\put(2206,-8251){\makebox(0,0)[lb]{\smash{{\SetFigFont{20}{24.0}{\familydefault}{\mddefault}{\updefault}{\color[rgb]{0,0,0}$\tq{4(3)}$}%
}}}}
\put(5986,-8251){\makebox(0,0)[lb]{\smash{{\SetFigFont{20}{24.0}{\familydefault}{\mddefault}{\updefault}{\color[rgb]{0,0,0}$\tq{4(1)}$}%
}}}}
\put(-2339,-14236){\makebox(0,0)[lb]{\smash{{\SetFigFont{20}{24.0}{\familydefault}{\mddefault}{\updefault}{\color[rgb]{0,0,0}$\tq{5(3)}$}%
}}}}
\put(1216,-14236){\makebox(0,0)[lb]{\smash{{\SetFigFont{20}{24.0}{\familydefault}{\mddefault}{\updefault}{\color[rgb]{0,0,0}$\tq{5(1)}$}%
}}}}
\put(-4184,-4651){\makebox(0,0)[lb]{\smash{{\SetFigFont{20}{24.0}{\familydefault}{\mddefault}{\updefault}{\color[rgb]{0,0,0}$\tq{3(3)}$}%
}}}}
\put(-1799,-4651){\makebox(0,0)[lb]{\smash{{\SetFigFont{20}{24.0}{\familydefault}{\mddefault}{\updefault}{\color[rgb]{0,0,0}$\tq{3(1)}$}%
}}}}
\end{picture}%

%% file: 225g.pdf_t
\begin{picture}(0,0)%
\includegraphics{225g.pdf}%
\end{picture}%
\setlength{\unitlength}{4144sp}%
\begingroup\makeatletter\ifx\SetFigFont\undefined%
\gdef\SetFigFont#1#2#3#4#5{%
  \reset@font\fontsize{#1}{#2pt}%
  \fontfamily{#3}\fontseries{#4}\fontshape{#5}%
  \selectfont}%
\fi\endgroup%
\begin{picture}(17658,19878)(-4373,-14395)
\put(6436,-1006){\makebox(0,0)[lb]{\smash{{\SetFigFont{20}{24.0}{\familydefault}{\mddefault}{\updefault}{\color[rgb]{0,0,0}$\tq{2(2)}$}%
}}}}
\put(4771,2054){\makebox(0,0)[lb]{\smash{{\SetFigFont{20}{24.0}{\familydefault}{\mddefault}{\updefault}{\color[rgb]{0,0,0}$\Aq{\phi_1}$}%
}}}}
\put(3421,659){\makebox(0,0)[lb]{\smash{{\SetFigFont{20}{24.0}{\familydefault}{\mddefault}{\updefault}{\color[rgb]{0,0,0}$\Aq{\chi_1}$}%
}}}}
\put(4861,-1006){\makebox(0,0)[lb]{\smash{{\SetFigFont{20}{24.0}{\familydefault}{\mddefault}{\updefault}{\color[rgb]{0,0,0}$\tq{2(3)}$}%
}}}}
\put(7921,-1006){\makebox(0,0)[lb]{\smash{{\SetFigFont{20}{24.0}{\familydefault}{\mddefault}{\updefault}{\color[rgb]{0,0,0}$\tq{2(1)}$}%
}}}}
\put(4096,-7936){\makebox(0,0)[lb]{\smash{{\SetFigFont{20}{24.0}{\familydefault}{\mddefault}{\updefault}{\color[rgb]{0,0,0}$\tq{4(2)}$}%
}}}}
\put(2566,-7936){\makebox(0,0)[lb]{\smash{{\SetFigFont{20}{24.0}{\familydefault}{\mddefault}{\updefault}{\color[rgb]{0,0,0}$\tq{4(3)}$}%
}}}}
\put(5581,-7936){\makebox(0,0)[lb]{\smash{{\SetFigFont{20}{24.0}{\familydefault}{\mddefault}{\updefault}{\color[rgb]{0,0,0}$\tq{4(1)}$}%
}}}}
\put(181,2909){\makebox(0,0)[lb]{\smash{{\SetFigFont{20}{24.0}{\familydefault}{\mddefault}{\updefault}{\color[rgb]{0,0,0}$\Rq{1(2)}$}%
}}}}
\put(-1394,2594){\makebox(0,0)[lb]{\smash{{\SetFigFont{20}{24.0}{\familydefault}{\mddefault}{\updefault}{\color[rgb]{0,0,0}$\tq{1(3)}$}%
}}}}
\put(-2339,-14326){\makebox(0,0)[lb]{\smash{{\SetFigFont{20}{24.0}{\familydefault}{\mddefault}{\updefault}{\color[rgb]{0,0,0}$\tq{5(3)}$}%
}}}}
\put(-674,-14326){\makebox(0,0)[lb]{\smash{{\SetFigFont{20}{24.0}{\familydefault}{\mddefault}{\updefault}{\color[rgb]{0,0,0}$\tq{5(2)}$}%
}}}}
\put(946,-14326){\makebox(0,0)[lb]{\smash{{\SetFigFont{20}{24.0}{\familydefault}{\mddefault}{\updefault}{\color[rgb]{0,0,0}$\tq{5(1)}$}%
}}}}
\put(1621,2594){\makebox(0,0)[lb]{\smash{{\SetFigFont{20}{24.0}{\familydefault}{\mddefault}{\updefault}{\color[rgb]{0,0,0}$\tq{1(1)}$}%
}}}}
\put(-3824,-5011){\makebox(0,0)[lb]{\smash{{\SetFigFont{20}{24.0}{\familydefault}{\mddefault}{\updefault}{\color[rgb]{0,0,0}$\tq{3(3)}$}%
}}}}
\put(-3014,-5011){\makebox(0,0)[lb]{\smash{{\SetFigFont{20}{24.0}{\familydefault}{\mddefault}{\updefault}{\color[rgb]{0,0,0}$\tq{3(2)}$}%
}}}}
\put(-2204,-5011){\makebox(0,0)[lb]{\smash{{\SetFigFont{20}{24.0}{\familydefault}{\mddefault}{\updefault}{\color[rgb]{0,0,0}$\tq{3(1)}$}%
}}}}
\end{picture}%

%% file: 195a.pdf_t
\begin{picture}(0,0)%
\includegraphics{195a.pdf}%
\end{picture}%
\setlength{\unitlength}{4144sp}%
\begingroup\makeatletter\ifx\SetFigFont\undefined%
\gdef\SetFigFont#1#2#3#4#5{%
  \reset@font\fontsize{#1}{#2pt}%
  \fontfamily{#3}\fontseries{#4}\fontshape{#5}%
  \selectfont}%
\fi\endgroup%
\begin{picture}(14215,14116)(17441,-13805)
\put(24391,-2806){\makebox(0,0)[lb]{\smash{{\SetFigFont{20}{24.0}{\familydefault}{\mddefault}{\updefault}{\color[rgb]{0,0,0}$\tq{2(2)}$}%
}}}}
\put(22816,-2806){\makebox(0,0)[lb]{\smash{{\SetFigFont{20}{24.0}{\familydefault}{\mddefault}{\updefault}{\color[rgb]{0,0,0}$\tq{2(3)}$}%
}}}}
\put(25876,-2806){\makebox(0,0)[lb]{\smash{{\SetFigFont{20}{24.0}{\familydefault}{\mddefault}{\updefault}{\color[rgb]{0,0,0}$\tq{2(1)}$}%
}}}}
\put(29476,-3166){\makebox(0,0)[lb]{\smash{{\SetFigFont{20}{24.0}{\familydefault}{\mddefault}{\updefault}{\color[rgb]{0,0,0}$\tq{3(2)}$}%
}}}}
\put(28666,-3166){\makebox(0,0)[lb]{\smash{{\SetFigFont{20}{24.0}{\familydefault}{\mddefault}{\updefault}{\color[rgb]{0,0,0}$\tq{3(3)}$}%
}}}}
\put(31051,-3166){\makebox(0,0)[lb]{\smash{{\SetFigFont{20}{24.0}{\familydefault}{\mddefault}{\updefault}{\color[rgb]{0,0,0}$\tq{3(1)}$}%
}}}}
\put(20296,-10366){\makebox(0,0)[lb]{\smash{{\SetFigFont{20}{24.0}{\familydefault}{\mddefault}{\updefault}{\color[rgb]{0,0,0}$\tq{4(2)}$}%
}}}}
\put(18721,-10456){\makebox(0,0)[lb]{\smash{{\SetFigFont{20}{24.0}{\familydefault}{\mddefault}{\updefault}{\color[rgb]{0,0,0}$\tq{4(3)}$}%
}}}}
\put(22501,-10456){\makebox(0,0)[lb]{\smash{{\SetFigFont{20}{24.0}{\familydefault}{\mddefault}{\updefault}{\color[rgb]{0,0,0}$\tq{4(1)}$}%
}}}}
\put(27811,-10681){\makebox(0,0)[lb]{\smash{{\SetFigFont{20}{24.0}{\familydefault}{\mddefault}{\updefault}{\color[rgb]{0,0,0}$\tq{5(2)}$}%
}}}}
\put(26416,-10771){\makebox(0,0)[lb]{\smash{{\SetFigFont{20}{24.0}{\familydefault}{\mddefault}{\updefault}{\color[rgb]{0,0,0}$\tq{5(3)}$}%
}}}}
\put(29971,-10771){\makebox(0,0)[lb]{\smash{{\SetFigFont{20}{24.0}{\familydefault}{\mddefault}{\updefault}{\color[rgb]{0,0,0}$\tq{5(1)}$}%
}}}}
\put(23716,-6406){\makebox(0,0)[lb]{\smash{{\SetFigFont{20}{24.0}{\familydefault}{\mddefault}{\updefault}{\color[rgb]{0,0,0}$\WW{5}$}%
}}}}
\put(25246,-6406){\makebox(0,0)[lb]{\smash{{\SetFigFont{20}{24.0}{\familydefault}{\mddefault}{\updefault}{\color[rgb]{0,0,0}$\WW{1}$}%
}}}}
\put(26776,-6406){\makebox(0,0)[lb]{\smash{{\SetFigFont{20}{24.0}{\familydefault}{\mddefault}{\updefault}{\color[rgb]{0,0,0}$\WW{3}$}%
}}}}
\put(17641,-3166){\makebox(0,0)[lb]{\smash{{\SetFigFont{20}{24.0}{\familydefault}{\mddefault}{\updefault}{\color[rgb]{0,0,0}$\tq{1(3)}$}%
}}}}
\put(18451,-3166){\makebox(0,0)[lb]{\smash{{\SetFigFont{20}{24.0}{\familydefault}{\mddefault}{\updefault}{\color[rgb]{0,0,0}$\tq{1(2)}$}%
}}}}
\put(20026,-3166){\makebox(0,0)[lb]{\smash{{\SetFigFont{20}{24.0}{\familydefault}{\mddefault}{\updefault}{\color[rgb]{0,0,0}$\tq{1(1)}$}%
}}}}
\put(18136,164){\makebox(0,0)[lb]{\smash{{\SetFigFont{20}{24.0}{\familydefault}{\mddefault}{\updefault}{\color[rgb]{0,0,0}$\WW{3}$}%
}}}}
\put(19666,164){\makebox(0,0)[lb]{\smash{{\SetFigFont{20}{24.0}{\familydefault}{\mddefault}{\updefault}{\color[rgb]{0,0,0}$\WW{2}$}%
}}}}
\put(29161,164){\makebox(0,0)[lb]{\smash{{\SetFigFont{20}{24.0}{\familydefault}{\mddefault}{\updefault}{\color[rgb]{0,0,0}$\WW{1}$}%
}}}}
\put(30691,164){\makebox(0,0)[lb]{\smash{{\SetFigFont{20}{24.0}{\familydefault}{\mddefault}{\updefault}{\color[rgb]{0,0,0}$\WW{2}$}%
}}}}
\put(19711,-6406){\makebox(0,0)[lb]{\smash{{\SetFigFont{20}{24.0}{\familydefault}{\mddefault}{\updefault}{\color[rgb]{0,0,0}$\WW{4}$}%
}}}}
\put(18181,-6406){\makebox(0,0)[lb]{\smash{{\SetFigFont{20}{24.0}{\familydefault}{\mddefault}{\updefault}{\color[rgb]{0,0,0}$\WW{5}$}%
}}}}
\put(22186,-6406){\makebox(0,0)[lb]{\smash{{\SetFigFont{20}{24.0}{\familydefault}{\mddefault}{\updefault}{\color[rgb]{0,0,0}$\WW{4}$}%
}}}}
\put(29206,-6406){\makebox(0,0)[lb]{\smash{{\SetFigFont{20}{24.0}{\familydefault}{\mddefault}{\updefault}{\color[rgb]{0,0,0}$\WW{4}$}%
}}}}
\put(30736,-6406){\makebox(0,0)[lb]{\smash{{\SetFigFont{20}{24.0}{\familydefault}{\mddefault}{\updefault}{\color[rgb]{0,0,0}$\WW{5}$}%
}}}}
\put(20746,-13651){\makebox(0,0)[lb]{\smash{{\SetFigFont{20}{24.0}{\familydefault}{\mddefault}{\updefault}{\color[rgb]{0,0,0}$\WW{5}$}%
}}}}
\put(20701,-7171){\makebox(0,0)[lb]{\smash{{\SetFigFont{20}{24.0}{\familydefault}{\mddefault}{\updefault}{\color[rgb]{0,0,0}$\WW{2}$}%
}}}}
\put(26641,-7216){\makebox(0,0)[lb]{\smash{{\SetFigFont{20}{24.0}{\familydefault}{\mddefault}{\updefault}{\color[rgb]{0,0,0}$\WW{3}$}%
}}}}
\put(28261,-7216){\makebox(0,0)[lb]{\smash{{\SetFigFont{20}{24.0}{\familydefault}{\mddefault}{\updefault}{\color[rgb]{0,0,0}$\WW{1}$}%
}}}}
\put(29881,-7216){\makebox(0,0)[lb]{\smash{{\SetFigFont{20}{24.0}{\familydefault}{\mddefault}{\updefault}{\color[rgb]{0,0,0}$\WW{4}$}%
}}}}
\put(28306,-13741){\makebox(0,0)[lb]{\smash{{\SetFigFont{20}{24.0}{\familydefault}{\mddefault}{\updefault}{\color[rgb]{0,0,0}$\WW{2}$}%
}}}}
\put(22276,-13651){\makebox(0,0)[lb]{\smash{{\SetFigFont{20}{24.0}{\familydefault}{\mddefault}{\updefault}{\color[rgb]{0,0,0}$\WW{3}$}%
}}}}
\put(19216,-13651){\makebox(0,0)[lb]{\smash{{\SetFigFont{20}{24.0}{\familydefault}{\mddefault}{\updefault}{\color[rgb]{0,0,0}$\WW{1}$}%
}}}}
\end{picture}%

%% file: 705a.pdf_t
\begin{picture}(0,0)%
\includegraphics{705a.pdf}%
\end{picture}%
\setlength{\unitlength}{4144sp}%
\begingroup\makeatletter\ifx\SetFigFont\undefined%
\gdef\SetFigFont#1#2#3#4#5{%
  \reset@font\fontsize{#1}{#2pt}%
  \fontfamily{#3}\fontseries{#4}\fontshape{#5}%
  \selectfont}%
\fi\endgroup%
\begin{picture}(5340,5251)(-734,-4895)
\put(1126,209){\makebox(0,0)[lb]{\smash{{\SetFigFont{12}{14.4}{\familydefault}{\mddefault}{\updefault}{\color[rgb]{0,0,0}$\ta{1(1)}$}%
}}}}
\put(2566,209){\makebox(0,0)[lb]{\smash{{\SetFigFont{12}{14.4}{\familydefault}{\mddefault}{\updefault}{\color[rgb]{0,0,0}$\ta{1(2)}$}%
}}}}
\put(2566,-4831){\makebox(0,0)[lb]{\smash{{\SetFigFont{12}{14.4}{\familydefault}{\mddefault}{\updefault}{\color[rgb]{0,0,0}$\ta{3(1)}$}%
}}}}
\put(1126,-4831){\makebox(0,0)[lb]{\smash{{\SetFigFont{12}{14.4}{\familydefault}{\mddefault}{\updefault}{\color[rgb]{0,0,0}$\ta{3(2)}$}%
}}}}
\put(4591,-1591){\makebox(0,0)[lb]{\smash{{\SetFigFont{12}{14.4}{\familydefault}{\mddefault}{\updefault}{\color[rgb]{0,0,0}$\ta{2(1)}$}%
}}}}
\put(4591,-3031){\makebox(0,0)[lb]{\smash{{\SetFigFont{12}{14.4}{\familydefault}{\mddefault}{\updefault}{\color[rgb]{0,0,0}$\ta{2(2)}$}%
}}}}
\put(-719,-3031){\makebox(0,0)[lb]{\smash{{\SetFigFont{12}{14.4}{\familydefault}{\mddefault}{\updefault}{\color[rgb]{0,0,0}$\ta{4(1)}$}%
}}}}
\put(-719,-1591){\makebox(0,0)[lb]{\smash{{\SetFigFont{12}{14.4}{\familydefault}{\mddefault}{\updefault}{\color[rgb]{0,0,0}$\ta{4(2)}$}%
}}}}
\end{picture}%

%% file: 255e.pdf_t
\begin{picture}(0,0)%
\includegraphics{255e.pdf}%
\end{picture}%
\setlength{\unitlength}{4144sp}%
\begingroup\makeatletter\ifx\SetFigFont\undefined%
\gdef\SetFigFont#1#2#3#4#5{%
  \reset@font\fontsize{#1}{#2pt}%
  \fontfamily{#3}\fontseries{#4}\fontshape{#5}%
  \selectfont}%
\fi\endgroup%
\begin{picture}(2574,3271)(841,-3320)
\put(2077,-1366){\makebox(0,0)[lb]{\smash{{\SetFigFont{17}{20.4}{\familydefault}{\mddefault}{\updefault}{\color[rgb]{0,0,0}$\tb{x}$}%
}}}}
\put(856,-196){\makebox(0,0)[lb]{\smash{{\SetFigFont{17}{20.4}{\familydefault}{\mddefault}{\updefault}{\color[rgb]{0,0,0}$\ta{a}$}%
}}}}
\put(3241,-196){\makebox(0,0)[lb]{\smash{{\SetFigFont{17}{20.4}{\familydefault}{\mddefault}{\updefault}{\color[rgb]{0,0,0}$\ta{a'}$}%
}}}}
\put(856,-3256){\makebox(0,0)[lb]{\smash{{\SetFigFont{17}{20.4}{\familydefault}{\mddefault}{\updefault}{\color[rgb]{0,0,0}$\ta{b'}$}%
}}}}
\put(3241,-3256){\makebox(0,0)[lb]{\smash{{\SetFigFont{17}{20.4}{\familydefault}{\mddefault}{\updefault}{\color[rgb]{0,0,0}$\ta{b}$}%
}}}}
\end{picture}%

%% file: 355f.pdf_t
\begin{picture}(0,0)%
\includegraphics{355f.pdf}%
\end{picture}%
\setlength{\unitlength}{4144sp}%
\begingroup\makeatletter\ifx\SetFigFont\undefined%
\gdef\SetFigFont#1#2#3#4#5{%
  \reset@font\fontsize{#1}{#2pt}%
  \fontfamily{#3}\fontseries{#4}\fontshape{#5}%
  \selectfont}%
\fi\endgroup%
\begin{picture}(14859,4152)(841,-4162)
\put(6172,-1366){\makebox(0,0)[lb]{\smash{{\SetFigFont{17}{20.4}{\familydefault}{\mddefault}{\updefault}{\color[rgb]{0,0,0}$\tb{x}$}%
}}}}
\put(10267,-1366){\makebox(0,0)[lb]{\smash{{\SetFigFont{17}{20.4}{\familydefault}{\mddefault}{\updefault}{\color[rgb]{0,0,0}$\tb{x}$}%
}}}}
\put(14362,-1366){\makebox(0,0)[lb]{\smash{{\SetFigFont{17}{20.4}{\familydefault}{\mddefault}{\updefault}{\color[rgb]{0,0,0}$\tb{x}$}%
}}}}
\put(2077,-1366){\makebox(0,0)[lb]{\smash{{\SetFigFont{17}{20.4}{\familydefault}{\mddefault}{\updefault}{\color[rgb]{0,0,0}$\tb{x}$}%
}}}}
\put(10171,-4066){\makebox(0,0)[lb]{\smash{{\SetFigFont{17}{20.4}{\familydefault}{\mddefault}{\updefault}{\color[rgb]{0,0,0}$\tta{\eta_3}$}%
}}}}
\put(14311,-4066){\makebox(0,0)[lb]{\smash{{\SetFigFont{17}{20.4}{\familydefault}{\mddefault}{\updefault}{\color[rgb]{0,0,0}$\tta{\eta_4}$}%
}}}}
\put(1981,-4066){\makebox(0,0)[lb]{\smash{{\SetFigFont{17}{20.4}{\familydefault}{\mddefault}{\updefault}{\color[rgb]{0,0,0}$\tta{\eta_1}$}%
}}}}
\put(6076,-4066){\makebox(0,0)[lb]{\smash{{\SetFigFont{17}{20.4}{\familydefault}{\mddefault}{\updefault}{\color[rgb]{0,0,0}$\tta{\eta_2}$}%
}}}}
\put(856,-241){\makebox(0,0)[lb]{\smash{{\SetFigFont{17}{20.4}{\familydefault}{\mddefault}{\updefault}{\color[rgb]{0,0,0}$\ta{a}$}%
}}}}
\put(3286,-241){\makebox(0,0)[lb]{\smash{{\SetFigFont{17}{20.4}{\familydefault}{\mddefault}{\updefault}{\color[rgb]{0,0,0}$\ta{a'}$}%
}}}}
\put(4951,-241){\makebox(0,0)[lb]{\smash{{\SetFigFont{17}{20.4}{\familydefault}{\mddefault}{\updefault}{\color[rgb]{0,0,0}$\ta{a}$}%
}}}}
\put(7381,-241){\makebox(0,0)[lb]{\smash{{\SetFigFont{17}{20.4}{\familydefault}{\mddefault}{\updefault}{\color[rgb]{0,0,0}$\ta{a'}$}%
}}}}
\put(9046,-241){\makebox(0,0)[lb]{\smash{{\SetFigFont{17}{20.4}{\familydefault}{\mddefault}{\updefault}{\color[rgb]{0,0,0}$\ta{a'}$}%
}}}}
\put(11476,-241){\makebox(0,0)[lb]{\smash{{\SetFigFont{17}{20.4}{\familydefault}{\mddefault}{\updefault}{\color[rgb]{0,0,0}$\ta{a}$}%
}}}}
\put(901,-3211){\makebox(0,0)[lb]{\smash{{\SetFigFont{17}{20.4}{\familydefault}{\mddefault}{\updefault}{\color[rgb]{0,0,0}$\ta{b}$}%
}}}}
\put(4996,-3211){\makebox(0,0)[lb]{\smash{{\SetFigFont{17}{20.4}{\familydefault}{\mddefault}{\updefault}{\color[rgb]{0,0,0}$\ta{c}$}%
}}}}
\put(7381,-3211){\makebox(0,0)[lb]{\smash{{\SetFigFont{17}{20.4}{\familydefault}{\mddefault}{\updefault}{\color[rgb]{0,0,0}$\ta{b}$}%
}}}}
\put(9091,-3211){\makebox(0,0)[lb]{\smash{{\SetFigFont{17}{20.4}{\familydefault}{\mddefault}{\updefault}{\color[rgb]{0,0,0}$\ta{b}$}%
}}}}
\put(13141,-241){\makebox(0,0)[lb]{\smash{{\SetFigFont{17}{20.4}{\familydefault}{\mddefault}{\updefault}{\color[rgb]{0,0,0}$\ta{a'}$}%
}}}}
\put(15571,-241){\makebox(0,0)[lb]{\smash{{\SetFigFont{17}{20.4}{\familydefault}{\mddefault}{\updefault}{\color[rgb]{0,0,0}$\ta{a}$}%
}}}}
\put(13186,-3211){\makebox(0,0)[lb]{\smash{{\SetFigFont{17}{20.4}{\familydefault}{\mddefault}{\updefault}{\color[rgb]{0,0,0}$\ta{c}$}%
}}}}
\put(15571,-3211){\makebox(0,0)[lb]{\smash{{\SetFigFont{17}{20.4}{\familydefault}{\mddefault}{\updefault}{\color[rgb]{0,0,0}$\ta{b}$}%
}}}}
\put(11386,-3211){\makebox(0,0)[lb]{\smash{{\SetFigFont{17}{20.4}{\familydefault}{\mddefault}{\updefault}{\color[rgb]{0,0,0}$\ta{c}$}%
}}}}
\put(3196,-3211){\makebox(0,0)[lb]{\smash{{\SetFigFont{17}{20.4}{\familydefault}{\mddefault}{\updefault}{\color[rgb]{0,0,0}$\ta{c}$}%
}}}}
\end{picture}%

%% file: 1745d.pdf_t
\begin{picture}(0,0)%
\includegraphics{1745d.pdf}%
\end{picture}%
\setlength{\unitlength}{4144sp}%
\begingroup\makeatletter\ifx\SetFigFont\undefined%
\gdef\SetFigFont#1#2#3#4#5{%
  \reset@font\fontsize{#1}{#2pt}%
  \fontfamily{#3}\fontseries{#4}\fontshape{#5}%
  \selectfont}%
\fi\endgroup%
\begin{picture}(20275,6268)(19028,-1025)
\put(22411,-961){\makebox(0,0)[lb]{\smash{{\SetFigFont{12}{14.4}{\familydefault}{\mddefault}{\updefault}{\color[rgb]{0,0,0}$\tarm{(a)}$}%
}}}}
\put(34561,-961){\makebox(0,0)[lb]{\smash{{\SetFigFont{12}{14.4}{\familydefault}{\mddefault}{\updefault}{\color[rgb]{0,0,0}$\tarm{(b)}$}%
}}}}
\end{picture}%

%% file: 1755c.pdf_t
\begin{picture}(0,0)%
\includegraphics{1755c.pdf}%
\end{picture}%
\setlength{\unitlength}{4144sp}%
\begingroup\makeatletter\ifx\SetFigFont\undefined%
\gdef\SetFigFont#1#2#3#4#5{%
  \reset@font\fontsize{#1}{#2pt}%
  \fontfamily{#3}\fontseries{#4}\fontshape{#5}%
  \selectfont}%
\fi\endgroup%
\begin{picture}(19508,6111)(19021,-1039)
\put(22411,-961){\makebox(0,0)[lb]{\smash{{\SetFigFont{12}{14.4}{\familydefault}{\mddefault}{\updefault}{\color[rgb]{0,0,0}$\tarm{(a)}$}%
}}}}
\put(34561,-961){\makebox(0,0)[lb]{\smash{{\SetFigFont{12}{14.4}{\familydefault}{\mddefault}{\updefault}{\color[rgb]{0,0,0}$\tarm{(b)}$}%
}}}}
\put(19036,4889){\makebox(0,0)[lb]{\smash{{\SetFigFont{12}{14.4}{\familydefault}{\mddefault}{\updefault}{\color[rgb]{0,0,0}$\ta{a}$}%
}}}}
\put(22636,4889){\makebox(0,0)[lb]{\smash{{\SetFigFont{12}{14.4}{\familydefault}{\mddefault}{\updefault}{\color[rgb]{0,0,0}$\ta{a'}$}%
}}}}
\put(26236,4889){\makebox(0,0)[lb]{\smash{{\SetFigFont{12}{14.4}{\familydefault}{\mddefault}{\updefault}{\color[rgb]{0,0,0}$\ta{a''}$}%
}}}}
\put(31051,4889){\makebox(0,0)[lb]{\smash{{\SetFigFont{12}{14.4}{\familydefault}{\mddefault}{\updefault}{\color[rgb]{0,0,0}$\ta{a}$}%
}}}}
\put(34651,4889){\makebox(0,0)[lb]{\smash{{\SetFigFont{12}{14.4}{\familydefault}{\mddefault}{\updefault}{\color[rgb]{0,0,0}$\ta{a'}$}%
}}}}
\put(38251,4889){\makebox(0,0)[lb]{\smash{{\SetFigFont{12}{14.4}{\familydefault}{\mddefault}{\updefault}{\color[rgb]{0,0,0}$\ta{a''}$}%
}}}}
\put(36451,119){\makebox(0,0)[lb]{\smash{{\SetFigFont{12}{14.4}{\familydefault}{\mddefault}{\updefault}{\color[rgb]{0,0,0}$\ta{b}$}%
}}}}
\put(32851,119){\makebox(0,0)[lb]{\smash{{\SetFigFont{12}{14.4}{\familydefault}{\mddefault}{\updefault}{\color[rgb]{0,0,0}$\ta{c}$}%
}}}}
\put(24436,119){\makebox(0,0)[lb]{\smash{{\SetFigFont{12}{14.4}{\familydefault}{\mddefault}{\updefault}{\color[rgb]{0,0,0}$\ta{c}$}%
}}}}
\put(20836,119){\makebox(0,0)[lb]{\smash{{\SetFigFont{12}{14.4}{\familydefault}{\mddefault}{\updefault}{\color[rgb]{0,0,0}$\ta{b}$}%
}}}}
\end{picture}%

%% file: 1765c.pdf_t
\begin{picture}(0,0)%
\includegraphics{1765c.pdf}%
\end{picture}%
\setlength{\unitlength}{4144sp}%
\begingroup\makeatletter\ifx\SetFigFont\undefined%
\gdef\SetFigFont#1#2#3#4#5{%
  \reset@font\fontsize{#1}{#2pt}%
  \fontfamily{#3}\fontseries{#4}\fontshape{#5}%
  \selectfont}%
\fi\endgroup%
\begin{picture}(25493,6156)(27931,-1039)
\put(31321,-961){\makebox(0,0)[lb]{\smash{{\SetFigFont{12}{14.4}{\familydefault}{\mddefault}{\updefault}{\color[rgb]{0,0,0}$\tarm{(a)}$}%
}}}}
\put(40231,-961){\makebox(0,0)[lb]{\smash{{\SetFigFont{12}{14.4}{\familydefault}{\mddefault}{\updefault}{\color[rgb]{0,0,0}$\tarm{(b)}$}%
}}}}
\put(49321,-961){\makebox(0,0)[lb]{\smash{{\SetFigFont{12}{14.4}{\familydefault}{\mddefault}{\updefault}{\color[rgb]{0,0,0}$\tarm{(c)}$}%
}}}}
\put(27946,4934){\makebox(0,0)[lb]{\smash{{\SetFigFont{12}{14.4}{\familydefault}{\mddefault}{\updefault}{\color[rgb]{0,0,0}$\ta{a}$}%
}}}}
\put(31501,4934){\makebox(0,0)[lb]{\smash{{\SetFigFont{12}{14.4}{\familydefault}{\mddefault}{\updefault}{\color[rgb]{0,0,0}$\ta{a'}$}%
}}}}
\put(35056,4934){\makebox(0,0)[lb]{\smash{{\SetFigFont{12}{14.4}{\familydefault}{\mddefault}{\updefault}{\color[rgb]{0,0,0}$\ta{a''}$}%
}}}}
\put(36856,4934){\makebox(0,0)[lb]{\smash{{\SetFigFont{12}{14.4}{\familydefault}{\mddefault}{\updefault}{\color[rgb]{0,0,0}$\ta{a}$}%
}}}}
\put(40411,4934){\makebox(0,0)[lb]{\smash{{\SetFigFont{12}{14.4}{\familydefault}{\mddefault}{\updefault}{\color[rgb]{0,0,0}$\ta{a'}$}%
}}}}
\put(43966,4934){\makebox(0,0)[lb]{\smash{{\SetFigFont{12}{14.4}{\familydefault}{\mddefault}{\updefault}{\color[rgb]{0,0,0}$\ta{a''}$}%
}}}}
\put(45946,4934){\makebox(0,0)[lb]{\smash{{\SetFigFont{12}{14.4}{\familydefault}{\mddefault}{\updefault}{\color[rgb]{0,0,0}$\ta{a}$}%
}}}}
\put(49501,4934){\makebox(0,0)[lb]{\smash{{\SetFigFont{12}{14.4}{\familydefault}{\mddefault}{\updefault}{\color[rgb]{0,0,0}$\ta{a'}$}%
}}}}
\put(53056,4934){\makebox(0,0)[lb]{\smash{{\SetFigFont{12}{14.4}{\familydefault}{\mddefault}{\updefault}{\color[rgb]{0,0,0}$\ta{a''}$}%
}}}}
\put(51301,254){\makebox(0,0)[lb]{\smash{{\SetFigFont{12}{14.4}{\familydefault}{\mddefault}{\updefault}{\color[rgb]{0,0,0}$\ta{d}$}%
}}}}
\put(47746,254){\makebox(0,0)[lb]{\smash{{\SetFigFont{12}{14.4}{\familydefault}{\mddefault}{\updefault}{\color[rgb]{0,0,0}$\ta{b}$}%
}}}}
\put(42211,254){\makebox(0,0)[lb]{\smash{{\SetFigFont{12}{14.4}{\familydefault}{\mddefault}{\updefault}{\color[rgb]{0,0,0}$\ta{c}$}%
}}}}
\put(38656,254){\makebox(0,0)[lb]{\smash{{\SetFigFont{12}{14.4}{\familydefault}{\mddefault}{\updefault}{\color[rgb]{0,0,0}$\ta{d}$}%
}}}}
\put(29746,254){\makebox(0,0)[lb]{\smash{{\SetFigFont{12}{14.4}{\familydefault}{\mddefault}{\updefault}{\color[rgb]{0,0,0}$\ta{c}$}%
}}}}
\put(33301,254){\makebox(0,0)[lb]{\smash{{\SetFigFont{12}{14.4}{\familydefault}{\mddefault}{\updefault}{\color[rgb]{0,0,0}$\ta{b}$}%
}}}}
\end{picture}%

%% file: 1775a.pdf_t
\begin{picture}(0,0)%
\includegraphics{1775a.pdf}%
\end{picture}%
\setlength{\unitlength}{4144sp}%
\begingroup\makeatletter\ifx\SetFigFont\undefined%
\gdef\SetFigFont#1#2#3#4#5{%
  \reset@font\fontsize{#1}{#2pt}%
  \fontfamily{#3}\fontseries{#4}\fontshape{#5}%
  \selectfont}%
\fi\endgroup%
\begin{picture}(10207,8211)(31178,-2634)
\end{picture}%

%% file: 205.pdf_t
\begin{picture}(0,0)%
\includegraphics{205.pdf}%
\end{picture}%
\setlength{\unitlength}{4144sp}%
\begingroup\makeatletter\ifx\SetFigFont\undefined%
\gdef\SetFigFont#1#2#3#4#5{%
  \reset@font\fontsize{#1}{#2pt}%
  \fontfamily{#3}\fontseries{#4}\fontshape{#5}%
  \selectfont}%
\fi\endgroup%
\begin{picture}(17526,23717)(-4384,-18332)
\put(6391,-1726){\makebox(0,0)[lb]{\smash{{\SetFigFont{20}{24.0}{\familydefault}{\mddefault}{\updefault}{\color[rgb]{0,0,0}$\te{2}$}%
}}}}
\put(4051,-8611){\makebox(0,0)[lb]{\smash{{\SetFigFont{20}{24.0}{\familydefault}{\mddefault}{\updefault}{\color[rgb]{0,0,0}$\te{4}$}%
}}}}
\put(-539,-14236){\makebox(0,0)[lb]{\smash{{\SetFigFont{20}{24.0}{\familydefault}{\mddefault}{\updefault}{\color[rgb]{0,0,0}$\te{5}$}%
}}}}
\put( 46,2324){\makebox(0,0)[lb]{\smash{{\SetFigFont{20}{24.0}{\familydefault}{\mddefault}{\updefault}{\color[rgb]{0,0,0}$\te{1}$}%
}}}}
\put(-3149,-4966){\makebox(0,0)[lb]{\smash{{\SetFigFont{20}{24.0}{\familydefault}{\mddefault}{\updefault}{\color[rgb]{0,0,0}$\te{3}$}%
}}}}
\end{picture}%